\documentclass[10pt, a4paper]{article}

\usepackage{natbib}
\bibpunct{[}{]}{,}{n}{}{;}

\usepackage{amsmath}
\usepackage{amsthm}
\usepackage{amssymb}
\usepackage{amsfonts}

\theoremstyle{plain}
\newtheorem{theorem}{Theorem}[section]
\newtheorem{lemma}[theorem]{Lemma}

\theoremstyle{definition}
\newtheorem{problem}{Problem}

\theoremstyle{remark}
\newtheorem{remark}{Remark}

\usepackage{csquotes}
\usepackage{array, multirow}
\usepackage{enumerate}
\usepackage{url}

\renewcommand{\setminus}{\smallsetminus}

\newcommand{\cpl}{\mathbb{P}^1}
\newcommand{\cpp}{\mathbb{P}^2}
\newcommand{\Aut}{\operatorname{Aut}}
\newcommand{\ord}{\operatorname{ord}}

\begin{document}

\title{On the Alexander Invariants of Trigonal Curves
}

\author{Melih \"U\c{c}er}




\maketitle

\begin{abstract}
\noindent
We show that most of the genus-zero subgroups of the braid group $\mathbb{B}_3$ (which are roughly the braid monodromy groups of the trigonal curves on the Hirzebruch surfaces) are irrelevant as far as the Alexander invariant is concerned: there is a very restricted class of \enquote{primitive} genus-zero subgroups such that these subgroups and their genus-zero intersections determine all the Alexander invariants. Then, we classify the primitive subgroups in a special subclass. This result implies the known classification of the dihedral covers of irreducible trigonal curves.
\end{abstract}

\section{Introduction}

In this paper, all varieties are over $\mathbb{C}$. Specifically, $\mathbb{P}^1$ denotes the Riemann sphere $\mathbb{C}\cup \{\infty\}$. Moreover, let $\Lambda := \mathbb{Z}[t, t^{-1}]$, and let $\Gamma := PSL(2, \mathbb{Z})$.

\subsection{Motivation}

Let $C\subset\cpp$ be a curve. The fundamental group $\pi_1(\cpp\setminus C)$ is an important invariant of $C$. It has been subject of interest since Zariski \citep{zariski}, yet its structure is still not well understood in general. As the singularities of $C$ grow, $\pi_1(\cpp\setminus C)$ gets more complicated. A precise statement in the direction of this principle is that under certain upper bounds on the singularities, $\pi_1(\cpp\setminus C)$ is abelian \citep{nori}. In this case, $\pi_1(\cpp\setminus C) = H_1(\cpp\setminus C)$, hence it is known due to Poincar\'{e} duality.

Another important invariant is the (conventional) \emph{Alexander invariant} $A_C^\mathrm{c}$ defined as follows: Let $d$ be the degree of $C$; then there is a canonical epimorphism $\operatorname{lk}\colon \pi_1(\cpp\setminus C) \twoheadrightarrow \mathbb{Z}_d$ which takes a loop to its linking coefficient with $C$. Then,
\begin{equation*}
A_C^\mathrm{c} := K\slash K', \qquad K := \operatorname{Ker}(\operatorname{lk}).
\end{equation*}
There is a canonical $\Lambda$-module structure on $A_C^\mathrm{c}$, where $t$ acts as conjugation by an element in $\operatorname{lk}^{-1}(1)$. Note that $t^d = 1$. Even though the Alexander invariant is simpler than the fundamental group, its structure is not fully known in general either. However, there are some general results which show that the Alexander invariant of a plane curve is significantly more restricted than that of knots.

The \emph{Alexander polynomial} $\Delta_C(t)$ is defined as the order of the $\mathbb{C}[t, t^{-1}]$-module $A_C^\mathrm{c}\otimes\mathbb{C}$, i.e. it is the characteristic polynomial of the $t$-action on the underlying vector space of $A_C^\mathrm{c}\otimes\mathbb{C}$. In fact, $\Delta_C(t)$ is the only isomorphism invariant of $A_C^\mathrm{c}\otimes\mathbb{C}$, because $t^d=1$ implies that all Jordan blocks of the $t$-action have size 1. The roots of $\Delta_C(t)$ are roots of unity with orders dividing $d$ (also because of $t^d=1$). In contrast, the roots of $\Delta_K(t)$ need not be roots of unity for a knot $K$. Moreover, if $C$ is irreducible, the order of a root of $\Delta_C(t)$ cannot be a prime power \citep{zariski-alexander}. Another general theorem on the Alexander polynomial is the following upper bound \citep{libgober}: Let $\{L_1, L_2, \ldots, L_n\}$ be the links cut by the singularities of $C$. Then,
\begin{equation*}
\Delta_C(t) \mid \prod_{1\leq i\leq m} \Delta_{L_i}(t).
\end{equation*}
Note that this theorem also illustrates the principle that these invariants get more complicated as the singularities of $C$ grow. 

In addition to these general theorems, there is a number of particular curves $C$ for which the invariants have been computed. For example, $\pi_1(\cpp\setminus C)$ is known for all $C$ with $d\leq 5$ \citep[see][]{degtyarev-quintics}. On the other hand, there are formulae which express $\Delta_C(t)$ in terms of the cohomology of certain linear systems, thus $\Delta_C(t)$ can be computed without topological methods \citep[e.g.][]{oka-review}. The textbook \citep{dimca} is a good source of information on this subject. 

\subsection{Monodromy Alexander Invariant}

In this paper, we study the so-called \emph{monodromy Alexander invariants} of trigonal curves, based on the results of Degtyarev \citep{degtyarev} on the braid monodromy of these curves.

Let $\Sigma = \Sigma_m$ be a Hirzebruch surface, let $E\subset\Sigma$ denote the exceptional section, and let $C\subset\Sigma$ be a curve. Note that the context of Hirzebruch surfaces is more general than that of $\cpp$, since $\cpp$ blown up at a point is $\Sigma_1$. Let $C$ be $n$-gonal, i.e. $C$ intersects a generic fiber of the ruling at $n$ points, and let the \emph{degree} of $C$ refer to $d := nm+[C\cdot E]$. Then, there is a canonical linking coefficient epimorphism $\operatorname{lk}\colon \pi_1(\Sigma\setminus (C\cup E)) \twoheadrightarrow \mathbb{Z}_d$ and the conventional Alexander invariant $A_C^\mathrm{c}$ is defined in terms of $\operatorname{lk}$ as in the case of plane curves. For a generic fiber $F_\infty$ of the ruling, there is a canonical epimorphism $\overline{\operatorname{lk}}\colon \pi_1(\Sigma\setminus (C\cup E\cup F_\infty)) \twoheadrightarrow \mathbb{Z}$, and $A_C^\mathrm{c}$ can be equivalently defined in terms of $\overline{\operatorname{lk}}$ instead of $\operatorname{lk}$, since there is an induced isomorphism $\operatorname{Ker}(\overline{\operatorname{lk}}) = \operatorname{Ker}(\operatorname{lk})$. Consider a plane curve $C\subset\cpp$ and let $C'\subset\Sigma_1$ be its proper transform; then $A_C^\mathrm{c} = A_{C'}^\mathrm{c}$.

Consider the ruling $\Sigma\rightarrow B\cong \cpl$. For the image $b_\infty\in B$ of $F_\infty$, the restricted projection $\Sigma\setminus (C\cup E\cup F_\infty)\rightarrow B\setminus\{b_\infty\}$ is topologically a fiber bundle away from finitely many singular fibers. Let $F_0 \neq F_\infty$ be another generic fiber of the ruling, then $F^\circ := F_0\setminus (C\cup E)$ is a fiber of this bundle. Clearly, $F^\circ$ is homeomorphic to a disk with $n$ punctures, so that $\pi_1(F^\circ)\cong F_n$, the free group on $n$ generators. Let $\{b_1, b_2, \ldots, b_k\}\subset B$ denote the image of the singular fibers. Then, there is a monodromy action $\pi_1(B\setminus \{b_\infty, b_1, b_2, \ldots, b_k\}) \rightarrow \Aut(\pi_1(F^\circ))\cong\Aut(F_n)$, whose image $M_C$ is called the \emph{braid monodromy group} of $C$. With some semi-standard choices, one has $M_C\subset\mathbb{B}_n\cdot\operatorname{Inn}(F_n)\subset\Aut(F_n)$ (see Section~\ref{sec:braid}). But since the choices are not unique, $M_C$ is well-defined only up to conjugation by $\mathbb{B}_n$. The Zariski-van Kampen theorem states
\begin{equation*}
\pi_1(\Sigma\setminus (C\cup E\cup F_\infty)) = F_n\slash\langle \alpha = m(\alpha) \mid \alpha\in F_n,\ m\in M_C\rangle.
\end{equation*}
Let $\operatorname{u}\colon F_n\twoheadrightarrow\mathbb{Z}$ be such that $\operatorname{u}$ maps each generator to 1. Then, $\operatorname{u}$ is the composition of $\overline{\operatorname{lk}}$ with the quotient epimorphism $F_n\twoheadrightarrow \pi_1(\Sigma\setminus (C\cup E\cup F_\infty))$.

Consider the $\Lambda$-module $A_n$ defined in terms of $\operatorname{u}$ in the same way that $A_C^\mathrm{c}$ is defined in terms of $\overline{\operatorname{lk}}$:
\begin{equation*}
A_n := K_n\slash K_n', \qquad K_n:= \operatorname{Ker}(\operatorname{u}).
\end{equation*}
Here, $M_C$ acts on $A_n\cong \Lambda^{n-1}$ via the \emph{Burau representation}, namely the induced action $\mathbb{B}_n\cdot\operatorname{Inn}(F_n)\rightarrow \Aut(A_n) = GL(n-1, \Lambda)$. The Zariski-van Kampen theorem motivates the definition of the \emph{monodromy Alexander invariant $A_C$}:
\begin{equation*}
A_C = \Lambda^{n-1}\slash\langle u = m(u) \mid u\in \Lambda^{n-1},\ m\in M_C\rangle.
\end{equation*}
There is a canonical epimorphism $A_C\twoheadrightarrow A_C^\mathrm{c}$, which is often (though not always) an isomorphism \citep{degtyarev}. Thus, as long as \enquote{upper bounds} are concerned, it suffices to classify the monodromy Alexander invariants. On the other hand, the latter is easier to compute than the conventional invariant, because it depends only on the image $H_C$ of $M_C\rightarrow GL(n-1, \Lambda)$, which we call the \emph{Burau monodromy group} of $C$. Note that $H_C$ is well-defined up to conjugacy as a subgroup of the image $\mathrm{Bu}_n$ of the Burau representation, which we call the \emph{Burau group}. In fact, for any subgroup $H\subset\mathrm{Bu}_n$, we can define
\begin{equation*}
A(H) := \Lambda^{n-1}\slash V(H), \qquad V(H) := \langle (h-1)\cdot u \mid u\in \Lambda^{n-1}, h\in H\rangle.
\end{equation*}
Consequently, $A_C = A(H_C)$. Clearly, the ambiguity in $H_C$ does not affect $A_C$ up to isomorphism.

From now on, we consider only the trigonal curves (the case $n=3$), whose Burau monodromy groups are almost completely characterized in Theorem~\ref{thm:degt-mdmy} below. This case is a borderline: the case $n\leq 2$ is quite easy, and the case $n\geq 4$ appears very difficult as of now. We ignore the very special case of \emph{isotrivial} trigonal curves, which have constant $j$-invariant on all fibers. There is a canonical epimorphism $\operatorname{c}\colon \mathrm{Bu}_3 \rightarrow \Gamma$ which is roughly defined by evaluation of a matrix at $t=-1$ (see Section~\ref{sec:braid}). A finite-index subgroup $H\subset\mathrm{Bu}_3$ is called \emph{genus-zero} if $\operatorname{c}(H)$ is genus-zero as a subgroup of the modular group $\Gamma$.
\begin{theorem}[Degtyarev \citep{degtyarev}]
\label{thm:degt-mdmy}
Let $C\subset\Sigma$ be a non-isotrivial trigonal curve; then, $H_C$ is genus-zero. For a partial converse, let $H\subset\mathbb{B}_3\subset\mathrm{Bu}_3$ be a genus-zero subgroup. Then, there is a non-isotrivial trigonal curve $C$ such that $H = H_C$.
\end{theorem}
In view of Theorem~\ref{thm:degt-mdmy}, the main question considered in this work can be approximately formulated as follows:
\begin{problem}
\label{prob:main}
Classify the modules of the form $A(H)$ for genus-zero $H\subset\mathrm{Bu}_3$.
\end{problem}

Degtyarev \citep{degtyarev} gave some partial answers to Problem~\ref{prob:main}. For example, he determined all possible eigenvalues of the $t$-action on the underlying vector space of $A(H)\otimes \mathbb{Q}$ and $A(H)\otimes \mathbb{F}_p$ for any prime $p$. For another example, he determined abelian groups of the form $A(H)\slash (t+1)$. This second example is the classification of the dihedral covers of trigonal curves.

\subsection{Principal Results}
\label{sec:result}

Given $H\subset\mathrm{Bu}_3$, the module $A(H)$ is equipped with an epimorphism $\Lambda^2\twoheadrightarrow A(H)$, which is always understood but usually omitted from notation. Conversely, given any module $A$ with an epimorphism $\phi\colon \Lambda^2\twoheadrightarrow A$, we can define the subgroup
\begin{equation*}
H(A) = H(\phi) := \{h\in\mathrm{Bu}_3\mid (h-1)\cdot\Lambda^2 \subset \operatorname{Ker}(\phi) \}.
\end{equation*}
Two epimorphisms $\phi_1, \phi_2$ are called \emph{Burau equivalent} if there exists $b\in\mathrm{Bu}_3$ such that $\operatorname{Ker}(\phi_1) = b\cdot \operatorname{Ker}(\phi_2)$. In this case, $H(\phi_1)$ and $H(\phi_2)$ are conjugate in $\mathrm{Bu}_3$.

For a maximal ideal $\mathfrak{m}\subset\Lambda$, the word \emph{$\mathfrak{m}$-local} refers to a non-trivial $\Lambda$-module annihilated by $\mathfrak{m}^n$ for sufficiently large $n$. Note that any maximal ideal $\mathfrak{m}$ is given in the form $\mathfrak{m} = \langle p, \psi(t) \rangle$ for a prime $p$ and a polynomial $\psi(t)$ which is irreducible modulo~$p$. Also note that any \emph{local} module is finite. Our Theorem~\ref{thm:setup} below suggests that Problem~\ref{prob:main} reduces to the classification of the genus-zero subgroups of the form $H(A)$ for local modules $A$. Then, our main Theorem~\ref{thm:main} classifies these subgroups for $\mathfrak{m}$-local $A$ for a special class of maximal ideals $\mathfrak{m}$. Only one subgroup from each conjugacy class is shown in the classification, since conjugate subgroups come from \emph{Burau equivalent} modules.
\begin{theorem}
\label{thm:setup}
Let $H\subset\mathrm{Bu}_3$ be a subgroup, and let $A$ be a module equipped with an epimorphism $\Lambda^2\twoheadrightarrow A$.
\begin{enumerate}
\item\label{thm:setup:closure} We have $H\subset H(A(H)), \enskip A(H) = A(H(A(H))), \enskip H(A) = H(A(H(A)))$.
\item\label{thm:setup:fin-rep} If $H(A)$ is finite-index, there is a finite quotient $A'$ of $A$ with $H(A) = H(A')$.
\item\label{thm:setup:decompose} If $A$ is finite, there is a decomposition $A = \bigoplus A_\mathfrak{m}$ into local modules $A_\mathfrak{m}$.
\item\label{thm:setup:intersect} If $A = \bigoplus A_i$ for some modules $A_i$, then $ H(A) = \bigcap H(A_i)$.
\end{enumerate}
\end{theorem}
It appears that very few intersections of the form $H(A_1)\cap H(A_2)$ are genus-zero, contributing to the list of Alexander invariants of trigonal curves \citep{degtyarev}.
\begin{theorem}[Main]
\label{thm:main}
Let $\mathfrak{m} := \langle p, t+1\rangle\subset\Lambda$ with $p\neq 2$, and let $A$ be an $\mathfrak{m}$-local module equipped with an epimorphism $\Lambda^2\twoheadrightarrow A$. If $H(A)\subset\mathrm{Bu}_3$ is genus-zero, then it is in one of the conjugacy classes listed in Table~\ref{tab:groups} on page~\pageref{tab:groups}.
\end{theorem}

The case $\mathfrak{m} = \langle p, t+1\rangle$ is likely to be more difficult than the case $t+1\not\in\mathfrak{m}$, which we intend to consider in a forthcoming paper. Moreover, this case is related to the dihedral covers of trigonal curves, i.e. cyclic covers of certain elliptic surfaces \citep[see][]{degtyarev}. The subgroups $H(A)$ which appear as part of the classification of the dihedral covers in loc. cit., except for $p=2$, are precisely those of depth~2 in Table~\ref{tab:groups}. We exclude the case $p=2$ since it would likely involve much more computation; on the other hand, $p=2$ does not appear for irreducible trigonal curves \citep{degtyarev}.


\subsection{The Contents of the Paper}

In Section~\ref{sec:groups}, we cite a few properties of the groups $\Gamma$, $\mathbb{B}_3$, and $\mathrm{Bu}_3$. In Section~\ref{sec:reduction}, we first give an alternative description of $H(A)$. Then, we prove Theorem~\ref{thm:setup} and establish restrictions on $H(A)$ for local $A$. Finally in Section~\ref{sec:proof}, we prove the main Theorem~\ref{thm:main}. Also, Table~\ref{tab:groups} and its explanation are given in Section~\ref{sec:table}, after necessary terminology has been introduced.

\paragraph{Acknowledgement.}
I thank Prof. Degtyarev under whose supervision I completed this work. He introduced me to the subject and constantly encouraged me during the preparation of this paper.

\section{Preliminaries}
\label{sec:groups}

This section contains necessary preliminary information on the modular group $\Gamma$, the braid group $\mathbb{B}_3$, and the Burau representation $\mathbb{B}_3\rightarrow \mathrm{Bu}_3\subset GL(2, \Lambda)$. The content of this section is completely standard; one can consult the classical sources \citep{serre-arithmetic, artin, burau}.

\subsection{The Modular Group $\Gamma$}
\label{sec:modular}

The modular group is often considered together with its left action on the complex upper half plane $\mathbb{H}$ via the inclusion $\Gamma \subset PSL(2, \mathbb{R}) = \Aut(\mathbb{H})$. Explicitly, the action of a matrix $\begin{pmatrix} a & b\\ c & d \end{pmatrix}$ is $z \mapsto \frac{az+b}{cz+d}$. This $\Gamma$-action is discrete and almost free: there are only two orbits for which the stabilizer is nontrivial, but the stabilizer is finite for these two orbits as well. Namely, the stabilizer of $\omega:= \frac{1+\sqrt{3}i}{2}$ is generated by $X := \begin{pmatrix} 0 & 1\\ -1 & 1 \end{pmatrix}$, i.e.~$z\mapsto \frac{1}{1-z}$; and the stabilizer of $i$ is  generated by $Y :=\begin{pmatrix} 0 & -1\\ 1 & 0 \end{pmatrix}$, i.e.~$z\mapsto -\frac{1}{z}$. A very classic theorem states
\begin{equation}
\label{eq:gamma-gen-rel}
\Gamma = \langle X, Y \mid X^3 = Y^2 = 1 \rangle.
\end{equation}
Hence, the abelianization of $\Gamma$ is isomorphic to $\mathbb{Z}_6$; we fix the abelianization $\operatorname{ab}\colon \Gamma\twoheadrightarrow \mathbb{Z}_6$ such that $\operatorname{ab}(X) = 2$ (note that one necessarily has $\operatorname{ab}(Y) = 3$).

\subsubsection{The Modular Curves}

Since the $\Gamma$-action on $\mathbb{H}$ is discrete and almost free, for any subgroup $K\subset\Gamma$, the quotient space $K\backslash \mathbb{H}$ naturally admits a Riemann surface structure (it also admits an orbifold structure, but we do not use this language explicitly). In particular, $\Gamma \backslash \mathbb{H} \cong \mathbb{C}$. We adopt Kodaira's normalization which fixes an identification $\Gamma \backslash \mathbb{H} = \mathbb{C}$ by mapping the orbits of $\omega\in\mathbb{H}$ and $i\in \mathbb{H}$ to $0\in\mathbb{C}$ and $1\in\mathbb{C}$, respectively.


Let $K\subset\Gamma$ be a finite-index subgroup. The \emph{modular curve} $X_K$ is a standard compactification of the Riemann surface $K\backslash\mathbb{H}$. In particular, $X_\Gamma = \mathbb{P}^1 = \mathbb{C}\cup\{\infty\}$. Any inclusion $K_1\subset K_2$ of subgroups clearly induces a non-constant (holomorphic) map $X_{K_1} \rightarrow X_{K_2}$ between the corresponding modular curves. For any $K$, the map $X_K\rightarrow X_\Gamma = \mathbb{P}^1$ is unramified outside the special points $\{0,1,\infty\}$. The conjugacy class of $K$ determines the map $X_K\rightarrow\mathbb{P}^1$ up to isomorphism. Conversely, the map $X_K\rightarrow\mathbb{P}^1$ determines $K$ up to conjugacy. The \emph{cusps} of $K$ are the points in $X_K$ which map to $\infty\in\mathbb{P}^1$. The \emph{width} of a cusp is the ramification index. The \emph{genus} of $K$, denoted by $g(K)$, is defined as that of $X_K$.

There is an immediate generalization of the construction above, which we find very useful. Let $E$ be a finite right $\Gamma$-set. The \emph{modular curve} $X_E$ is the disjoint union of the curves $X_K$ as $K$ varies over the stabilizers of distinct orbits in $E$. The $\Gamma$-set $E$ and the map $X_E\rightarrow\mathbb{P}^1$ determine each other up to isomorphism. The \emph{cusps} of $E$ are similarly defined. The notion of \emph{genus} $g(E)$ applies when $E$ is transitive, while that of \emph{Euler characteristic} $\chi(E)$ is meaningful in general. We denote the singleton $\Gamma$-set by $\{*\}$, as such $X_{\{*\}} = X_\Gamma = \mathbb{P}^1$. For any $\Gamma$-equivariant surjection $E_1\rightarrow E_2$, there is a covering $X_{E_1}\rightarrow X_{E_2}$ (possibly ramified). In the subsequent sections, we frequently deal with surjections of $\Gamma$-sets. Whenever we speak of a covering, it is possibly ramified.

\begin{remark}
\label{rem:genus}
Let $E_1, E_2$ be finite transitive $\Gamma$-sets with a $\Gamma$-equivariant surjection $E_1\rightarrow E_2$. If $g(E_1) = 0$, then $g(E_2) = 0$ as well. This is clear because there is a covering $X_{E_1}\rightarrow X_{E_2}$. Consequently, let $K_1\subset K_2\subset\Gamma$ be finite-index subgroups. If $g(K_1) = 0$, then $g(K_2) = 0$ as well.
\end{remark}


\subsubsection{The Standard CW-structures on the Modular Curves}

The \emph{terminal bipartite graph} $\mathrel{\bullet}\joinrel\relbar\joinrel\relbar\joinrel\relbar\joinrel\mathrel{\circ}$  is canonically embedded in $\mathbb{C}\subset \mathbb{P}^1 = X_{\{*\}}$ such that the black vertex goes to $0$, the white vertex goes to 1, and the edge goes to the real interval $[0,1]$. For any finite right $\Gamma$-set $E$, we denote the preimage of this graph under the map $X_E\rightarrow X_{\{*\}}$ by $\mathcal{S}_E$. In particular, we denote the terminal bipartite graph itself by $\mathcal{S}_{\{*\}}$. The notation $\mathcal{S}_K$ is similarly defined for finite-index subgroups $K\subset\Gamma$. Note that the restricted map $X_E\setminus \mathcal{S}_E \rightarrow X_{\{*\}}\setminus \mathcal{S}_{\{*\}}$ is unramified outside $\infty$, since $0, 1\in \mathcal{S}_{\{*\}}$. Thus, each component of $X_E\setminus \mathcal{S}_E$ is a 2-cell. Hence, $\mathcal{S}_E$ provides a CW decomposition of $X_E$. Clearly, each of the 2-cells contains exactly one cusp. Note that $\mathcal{S}_E$ is a ribbon graph in a natural way, since it is embedded in an oriented surface. By convention, we agree that the cyclic ordering of the edges is in the counter-clockwise direction. In fact, the ribbon graph $\mathcal{S}_E$ coincides with Grothendieck's \emph{dessins d'enfant} corresponding to the ramified covering $X_E\rightarrow\mathbb{P}^1$ \citep[see][]{grothendieck}. The preimage $\mathcal{F}$ of $\mathcal{S}_{\{*\}}$ under the map $\mathbb{H}\rightarrow \Gamma \backslash\mathbb{H} = \mathbb{C}$ is a tree \citep[e.g.][]{serre-trees}. Clearly, $\mathcal{F}$ has a black vertex at $\omega$ and a white vertex at $i$. Moreover, $\omega$ and $i$ are joined by an edge $e$. The $\Gamma$-action on $\mathbb{H}$ restricts to an action on $\mathcal{F}$. It is interesting that this action immediately shows~(\ref{eq:gamma-gen-rel}), by the Serre theory \citep[see][]{serre-trees}.

\medskip

\emph{The set of edges of $\mathcal{S}_E$ is a right $\Gamma$-set in a natural way, moreover it is isomorphic to $E$.} Consider the two loops $x, y$ in $\mathbb{P}^1\setminus\{0,1,\infty\}$ based at $\frac{1}{2}$, formed as counterclockwise circles of radius $\frac{1}{2}$ centered at 0 and 1, respectively. Then, the lifts of the path $x$ under the covering map $X_E\rightarrow \mathbb{P}^1$ define the action of $X$ on the set of edges of $\mathcal{S}_E$, while the lifts of the path $y$ define the action of $Y$. More explicitly, $X$ takes each edge to the next one among the edges sharing the same black vertex and $Y$ takes each edge to the next one among the edges sharing the same white vertex. Here, \enquote{next} refers to the cyclic order coming from the ribbon graph structure. These actions of $X$ and $Y$ uniquely extend to a right $\Gamma$-action. Thus, the action of $YX$ is described by the lifts of a certain loop formed by joining a clockwise circle around $\infty$ to $\frac{1}{2}$ along a path lying in the lower half plane, since $yx$ is homotopic to such a loop. Hence, the cusps are in bijection with $YX$-orbits.

The right $\Gamma$-action just described applies to $\mathcal{F}$ as well and it can be equivalently characterized as follows: The left $\Gamma$-action on the set of edges of $\mathcal{F}$ is free and transitive; hence, by identifying the edge $e$ with $1\in\Gamma$, one identifies this set with $\Gamma$. The right $\Gamma$-action comes from this identification. We now show the isomorphism between the set of edges of $\mathcal{S}_E$ and $E$. Clearly, one can assume that $E$ is transitive. Let $K$ be the stabilizer of any element of $E$ (well-defined up to conjugacy); then, one has $\mathcal{S}_E \cong \mathcal{S}_K$. On the other hand, $\mathcal{S}_K = K\backslash\mathcal{F}$; thus, the set of edges of $\mathcal{S}_K$ is identified with $K\backslash\Gamma$, which is isomorphic to $E$ as right $\Gamma$-sets.

\medskip

In light of the above, we introduce the following terminology for a right $\Gamma$-set $E$. The \emph{black vertices} in $E$ are the $X$-orbits, the \emph{white vertices} are the $Y$-orbits, the \emph{edges} are simply the elements of $E$, and the \emph{regions} are the $YX$-orbits. Then, the black and white vertices and the edges in $E$ are in bijection with those of $\mathcal{S}_E$, while the regions in $E$ are in bijection with the cusps of $E$, or equivalently, the components (2-cells) of $X_E\setminus \mathcal{S}_E$. Continuing to imitate the graph theory language, we say that a vertex in $E$ is \emph{monovalent} if it consists of a single element. Furthermore, we often speak of an equivariant surjection $E_1\rightarrow E_2$ of $\Gamma$-sets as a \emph{covering}. A covering takes vertices to vertices, regions to regions, etc. For vertices and regions, we speak of \emph{ramification}, whose meaning must be clear. For example, a vertex which is not monovalent is necessarily unramified. Similarly, the meaning of the \emph{degree} of a covering, or the notion of a \emph{regular} (\emph{Galois}) covering must be clear as well. We now give a formula for the Euler characteristic of a $\Gamma$-set.
\begin{lemma}[Euler Characteristic Formula]
\label{lem:genus}
Let $E$ be a finite right $\Gamma$-set. For any $\gamma\in \Gamma$, let $|E_\gamma|$ denote the number of $\gamma$-orbits in $E$ and $|E^\gamma|$ denote the number of $\gamma$-fixed elements in $E$.
Then
\begin{align*}
\chi(E) & = |E_X| + |E_Y| - |E| + |E_{YX}| \\
& = -  \frac{|E|}{6} + |E_{YX}| + \frac{2}{3}\cdot|E^X| +  \frac{1}{2}\cdot|E^Y|
\end{align*}
\end{lemma}
\begin{proof}
By definition, $\chi(E) = \chi(X_E)$. In the canonical CW-decomposition of $X_E$, the number of 0-cells (the black and white vertices) is $|E_X| + |E_Y|$, the number of 1-cells (the edges) is $|E|$ and the number of 2-cells is $|E_{YX}|$; this establishes the formula in the top line. For the bottom line, it is sufficient to observe that $|E_X| = \frac{|E|}{3} + \frac{2}{3}\cdot|E^X|$ and $|E_Y| = \frac{|E|}{2} + \frac{1}{2}\cdot|E^Y|$. This is because each $X$-orbit contains 1 or 3 elements and each $Y$-orbit contains 1 or 2 elements (since $X^3=Y^2=1$).
\end{proof}

\subsection{The Braid Groups and the Burau Representation \citep{artin, burau}}
\label{sec:braid}

Consider the free group $F_n$ as equipped with a fixed $n$-tuple $(s_1, s_2, \ldots, s_n)$ of generators. The \emph{braid group} $\mathbb{B}_n$ consists of those elements in $\Aut(F_n)$ which take each $s_i$ to a conjugate of some $s_j$ and which fix the product $s_1s_2\cdots s_n\in F_n$. Then, $\mathbb{B}_n$ is generated by $\sigma_1, \sigma_2, \ldots, \sigma_{n-1}$ which are defined by
\begin{equation*}
\sigma_i \colon s_i\mapsto s_is_{i+1}s_i^{-1}, \qquad s_{i+1}\mapsto s_i, \qquad s_j\mapsto s_j\ \text{for}\ j\neq i, i+1.
\end{equation*}
The action of $\mathbb{B}_n\cdot\operatorname{Inn}(F_n)\subset\Aut(F_n)$ on $F_n$ respects the epimorphism $\operatorname{u}\colon F_n\twoheadrightarrow\mathbb{Z}$ defined by $\operatorname{u}(s_i) = 1$. The \emph{Burau representation} $\mathbb{B}_n\cdot\operatorname{Inn}(F_n)\rightarrow GL(n-1, \Lambda)$ is the induced action on $A_n := \operatorname{Ker}(\operatorname{u})\slash\operatorname{Ker}(\operatorname{u})' \cong \Lambda^{n-1}$. Here, we identify $A_n$ with $\Lambda^{n-1}$ by matching the special basis $(s_1s_2^{-1}, s_2s_3^{-1}, \ldots, s_{n-1}s_n^{-1})$ of the former with the standard basis of the latter.

In the braid group $\mathbb{B}_3$, let $\mathbb{X} := \sigma_1\sigma_2$ and let $\mathbb{Y} := \sigma_1\sigma_2\sigma_1$. Then, one has
\begin{equation}
\label{eq:b3-gen}
\mathbb{B}_3 = \langle \mathbb{X}, \mathbb{Y}\mid \mathbb{X}^3 = \mathbb{Y}^2\rangle.
\end{equation}
 Explicitly written out,
\begin{equation*}
\begin{aligned}
\mathbb{X} \colon & s_1\mapsto s_1s_2s_1^{-1},
& \qquad &  s_2\mapsto s_1s_3s_1^{-1},
& \qquad &  s_3\mapsto s_1, \\
\mathbb{Y} \colon & s_1\mapsto s_1s_2s_3s_2^{-1}s_1^{-1},
& & s_2\mapsto s_1s_2s_1^{-1},
& & s_3\mapsto s_1,
\end{aligned}
\end{equation*}
and 
$\mathbb{X}^3 = \mathbb{Y}^2\colon s_i\mapsto (s_1s_2s_3)\cdot s_i\cdot (s_1s_2s_3)^{-1}$. Moreover,
\begin{equation}
\label{eq:faithful}
\text{the Burau representation}\ \mathbb{B}_3 \rightarrow \mathrm{Bu}_3\subset GL(2,\Lambda)\ \text{is faithful}.
\end{equation}
Hence, we identify $\mathbb{B}_3$ with its image and write
\begin{equation*}
\mathbb{X} = \begin{pmatrix} 0&- t\\ t & -t\end{pmatrix}, \qquad \mathbb{Y} = \begin{pmatrix} 0& -t\\ -t^2 & 0\end{pmatrix}.
\end{equation*}
Thus, $\mathbb{X}^3 = \mathbb{Y}^2 = t^3\cdot 1$. Clearly, $\mathrm{Bu}_3$ is generated by $\mathbb{X}$, $\mathbb{Y}$ and $t\cdot 1$. Then, one has $|\mathrm{Bu}_3:\mathbb{B}_3| = 3$ since
\begin{equation}
\label{eq:bu3>b3}
t\cdot 1\not\in\mathbb{B}_3.
\end{equation}

There is a canonical homomorphism $\operatorname{c}\times \operatorname{dg}\colon \mathrm{Bu}_3 \rightarrow \Gamma \times \mathbb{Z}$: the first component $\operatorname{c}$ is defined by the evaluation of a matrix at $t=-1$ followed by projectivization, and the second component is defined as $\operatorname{dg}(b) = \operatorname{deg}(\operatorname{det}(b))$ for any matrix $b$. Then,
\begin{equation*}
\begin{aligned}
& \operatorname{c}(\mathbb{X}) = X, \qquad && \operatorname{c}(\mathbb{Y}) = Y, \qquad && \operatorname{c}(t\cdot 1) = 1,\\
& \operatorname{dg}(\mathbb{X}) = 2, \qquad && \operatorname{dg}(\mathbb{Y}) = 3, \qquad && \operatorname{dg}(t\cdot 1) = 2.
\end{aligned}
\end{equation*}
It is easy to see that the image of $\mathbb{B}_3$ under $\operatorname{c}\times \operatorname{dg}$ consists of those pairs $(\gamma, n)$ for which $\operatorname{ab}(\gamma) \equiv n \pmod{6}$, which proves~(\ref{eq:bu3>b3}). It is also easy to see that the only relation between $\tilde{X} := (X, 2)$ and $\tilde{Y} := (Y, 3)$ in $\Gamma\times\mathbb{Z}$ is that $\tilde{X}^3 = \tilde{Y}^2$, because the only relation between $X$ and $Y$ in $\Gamma$ is $X^3=Y^2=1$. This observation shows that the only relation between $\mathbb{X}$ and $\mathbb{Y}$ in $\mathbb{B}_3$ is $\mathbb{X}^3 = \mathbb{Y}^2$ (see~(\ref{eq:b3-gen})), it shows~(\ref{eq:faithful}), and it shows that $\operatorname{c}\times \operatorname{dg}$ is injective, all at once.

We define the \emph{depth} $\operatorname{d}(H)$ of a finite-index subgroup $H\subset\mathrm{Bu}_3$ as the least integer such that $t^{\operatorname{d}(H)}\cdot 1\in H$. Equivalently, $\operatorname{d}(H) = |\operatorname{Ker}(\operatorname{c}):\operatorname{Ker}(\operatorname{c})\cap H|$.

\section{Reduction to Local Modules}
\label{sec:reduction}

Let $A$ be a $\Lambda$-module. Consider the following right $\mathrm{Bu}_3$-action on the set $A^2$ of pairs of elements:
\begin{equation}
\label{eq:action}
(a_1, a_2)\cdot\begin{pmatrix} x & y \\ z & w \end{pmatrix} = (x\cdot a_1 + z\cdot a_2, y\cdot a_1 + w\cdot a_2), \qquad \begin{pmatrix} x & y \\ z & w \end{pmatrix}\in\mathrm{Bu}_3.
\end{equation}
This $\mathrm{Bu}_3$-action restricts to the subset $\mathcal{E}(A) := \{(a_1, a_2)\in A^2 \mid \Lambda\cdot a_1 + \Lambda\cdot a_2 = A\}$ of generating pairs. This allows us to give an alternative description of $H(\Lambda^2\twoheadrightarrow A)$.
\begin{lemma}
\label{lem:stab}
Let $\phi\colon\Lambda^2\twoheadrightarrow A$ be an epimorphism, and let $e_1 := \phi\left({\footnotesize\begin{bmatrix} 1\\ 0 \end{bmatrix}}\right)$ and $e_2 := \phi\left({\footnotesize\begin{bmatrix} 0\\ 1 \end{bmatrix}}\right)$. Then, $H(\phi)\subset\mathrm{Bu}_3$ is the stabilizer of $(e_1, e_2)\in\mathcal{E}(A)$.
\end{lemma}
\begin{proof}
Let $b := \begin{pmatrix} x & y \\ z & w \end{pmatrix}\in H(\phi)$. By definition, ${\footnotesize\begin{bmatrix} x -1 \\ z \end{bmatrix}} = (b-1)\cdot{\footnotesize\begin{bmatrix} 1\\ 0 \end{bmatrix}} \in \operatorname{Ker}(\phi)$ and ${\footnotesize \begin{bmatrix} y \\ w-1 \end{bmatrix}} = (b-1)\cdot{\footnotesize\begin{bmatrix} 0\\ 1 \end{bmatrix}} \in \operatorname{Ker}(\phi)$. Thus, $(x-1)\cdot e_1 + z\cdot e_2 = \phi\left({\footnotesize\begin{bmatrix} x -1 \\ z \end{bmatrix}}\right) = 0$ and $y\cdot e_1 + (w-1)\cdot e_2 = \phi\left({\footnotesize\begin{bmatrix} y \\ w-1 \end{bmatrix}}\right) = 0$. Hence, $(e_1, e_2)\cdot b = (e_1, e_2)$. This proves that $H(\phi)$ is contained in the stabilizer. The reverse inclusion can be proved with exactly the same calculations, in the backward direction.
\end{proof}

We now prove item~\ref{thm:setup:fin-rep} of Theorem~\ref{thm:setup}. In fact, items~\ref{thm:setup:closure} and~\ref{thm:setup:intersect} are obvious, and item~\ref{thm:setup:decompose} is a general fact of algebra; thus we do not prove them.
\begin{proof}[Proof of Theorem~\ref{thm:setup}, item~\ref{thm:setup:fin-rep}]
Let $d:=\operatorname{d}(H(A))$. Since $t^d\cdot 1\in H(A)$, we have ${\footnotesize\begin{bmatrix} t^d-1\\ 0 \end{bmatrix}} \in \operatorname{Ker}(\Lambda^2\twoheadrightarrow A)$ and ${\footnotesize\begin{bmatrix} 0\\ t^d-1 \end{bmatrix}} \in \operatorname{Ker}(\Lambda^2\twoheadrightarrow A)$. Therefore, $A$ is a quotient of $\big(\Lambda\slash(t^d-1)\big)^2$. In particular, $A$ is finitely generated over $\mathbb{Z}$. Hence, for any finite set $S\subset A$ of nonzero elements, there is a positive integer $n$ such that $S\cap nA = \emptyset$. Note that the quotient module $A' := A\slash nA$ is finite.

Let $(e_1,e_2)\in\mathcal{E}(A)$ be as in Lemma~\ref{lem:stab}, let $\mathrm{O}$ denote the $\mathrm{Bu}_3$-orbit of $(e_1, e_2)$, and let $S$ be the set of all nonzero $s\in A$ such that $\mathrm{O}$ contains a pair of the form $(e_1+s, a)$ or $(a, e_2+s)$. Since $H(A)$ is finite-index, $\mathrm{O}$ is finite, hence $S$ is finite. Let $A'$ be a finite quotient of $A$ such that $S\cap \operatorname{Ker}(A\twoheadrightarrow A') = \emptyset$ and let $\mathrm{O}'\subset\mathcal{E}(A')$ be the orbit such that $\mathrm{O}\mapsto \mathrm{O}'$. Clearly $\mathrm{O}\cong \mathrm{O}'$, hence $H(A) = H(A')$.
\end{proof}

Let $A$ be a local module. We denote the set of $t$-orbits in $\mathcal{E}(A)$ by $\mathcal{C}(A)$, i.e. $(a_1, a_2)\sim(t^k\cdot a_1, t^k\cdot a_2)$. We denote the quotient map by $\operatorname{c}\colon \mathcal{E}(A)\rightarrow\mathcal{C}(A)$. Since $t\cdot 1\in\mathrm{Bu}_3$ generates the kernel of $\operatorname{c}\colon \mathrm{Bu}_3 \rightarrow \Gamma$ (see Section~\ref{sec:braid}), the $\mathrm{Bu}_3$-action in~(\ref{eq:action}) reduces to a $\Gamma$-action on $\mathcal{C}(A)$. This action is explicitly described as follows:
\begin{equation}
\label{eq:gamma-act}
\begin{aligned}
& \operatorname{c}(a_1,a_2)\cdot X && = && \operatorname{c}((a_1,a_2)\cdot(t^{-1}\mathbb{X})) && = && \operatorname{c}(a_2,-a_1-a_2),\\
& \operatorname{c}(a_1,a_2)\cdot Y && = && \operatorname{c}((a_1,a_2)\cdot(t^{-1}\mathbb{Y})) && = && \operatorname{c}(-ta_2,-a_1),\\
& \operatorname{c}(a_1,a_2)\cdot YX && = && \operatorname{c}(-ta_2,-a_1)\cdot X && = && \operatorname{c}(-a_1,ta_2+a_1).
\end{aligned}
\end{equation}
Note that an epimorphism $A_1\twoheadrightarrow A_2$ induces a covering $\mathcal{C}(A_1)\rightarrow\mathcal{C}(A_2)$. 

Let $\Omega\subset\mathcal{C}(A)$ be a $\Gamma$-orbit, then $\operatorname{c}^{-1}(\Omega)\subset\mathcal{E}(A)$ is a $\mathrm{Bu}_3$-orbit. We denote by $\mathcal{H}(\Omega)\subset\mathrm{Bu}_3$ the stabilizer of an arbitrary pair in $\operatorname{c}^{-1}(\Omega)$, thus it is well-defined up to conjugacy. Clearly, $\mathcal{H}(\Omega)$ is genus-zero if and only if $\Omega$ is genus-zero. Now, suppose that $A$ is equipped with an epimorphism $\Lambda^2\twoheadrightarrow A$, equivalently a distinguished pair $(e_1, e_2)\in\mathcal{E}(A)$. Let $\Omega_0$ be the orbit of $\operatorname{c}(e_1, e_2)\in\mathcal{C}(A)$. By Lemma~\ref{lem:stab}, $H(A)$ is equal to $\mathcal{H}(\Omega_0)$ up to conjugacy. Hence, for a proof of Theorem~\ref{thm:main}, it is enough to check $\mathfrak{m}$-local modules $A$ and orbits $\Omega\subset\mathcal{C}(A)$.

For a genus-zero orbit $\Omega\subset\mathcal{C}(A)$, the subgroup $\mathcal{H}(\Omega)$ is completely determined (up to conjugacy) by the so-called \emph{type specification} on $\Omega$ \citep[see][]{degtyarev}. Explicitly, let $d$ denote the order of the $t$-action on $A$; clearly $d = \operatorname{d}(\mathcal{H}(\Omega))$. For each monovalent vertex or region $\mathfrak{a}\subset\Omega$, consider an arbitrary pair $(a_1, a_2)\in\operatorname{c}^{-1}(\mathfrak{a})\subset\operatorname{c}^{-1}(\Omega)$ and let $\mathbb{L}$ denote $\mathbb{X}$, $\mathbb{Y}$, or $(\mathbb{Y}\mathbb{X})^{|\mathfrak{a}|}$ depending on whether $\mathfrak{a}$ is a black vertex, a white vertex, or a region, respectively. Then, there is an integer $k$ which satisfies $(a_1, a_2)\cdot \mathbb{L} = (t^k\cdot a_1, t^k\cdot a_2)$; let $\operatorname{k}(\mathfrak{a})\in\mathbb{Z}_d$ be the value of $k$ modulo $d$ which is unique and independent of the choice of $(a_1, a_2)$. The type specification is essentially the data which consists of $d$ and the collection of the values $\operatorname{k}(\mathfrak{a})$. 

\subsection{Local Modules}


In the rest, whenever $\mathfrak{m}\subset\Lambda$ refers to a particular maximal ideal, $\Bbbk$ denotes the residue field $\Lambda\slash \mathfrak{m}$. Let $A$ be an $\mathfrak{m}$-local module. By Nakayama's Lemma, a subset $\{a_1,a_2,\ldots,a_n\}\subset A$ generates $A$ if and only if its projection generates the vector space $A\otimes\Bbbk = A\slash \mathfrak{m}A$. Nakayama's Lemma applies because $A$ can be considered as a module over the local ring $\Lambda\slash \mathfrak{m}^n$. Note that we only consider modules with $\dim(A\otimes\Bbbk) \leq 2$, for $\mathcal{C}(A)$ is otherwise empty. In the case $\dim(A\otimes\Bbbk) = 1$, the module $A$ can be generated by one element, hence it is cyclic. In the case $\dim(A\otimes\Bbbk) = 2$, a pair $(a_1,a_2)$ is in $\mathcal{E}(A)$ if and only if $a_1$ and $a_2$ project to linearly independent nonzero vectors in $A\otimes\Bbbk$. We briefly call the modules in the latter class \emph{wheels}. We discuss the modules in the former class in a separate section.

\subsubsection{Modules with $\dim(A\otimes\Bbbk) = 1$}

Let $A$ be a module in this class and let $R$ denote the quotient (as a ring) of $\Lambda$ by the annihilator of $A$. Since $A$ is cyclic, it is isomorphic to $R$ as a $\Lambda$-module. Conversely, the quotient of $\Lambda$ by any ideal which contains some power of $\mathfrak{m}$ is a module in this class. Therefore, from now on, we only consider the rings $R$. Moreover, let the word \emph{ring} always refer to a ring of this type. Whenever we consider a ring $R$, we denote by $\mathfrak{m}$ the image of $\mathfrak{m}\subset\Lambda$ in $R$, then the unique maximal ideal $\mathfrak{m}\subset R$ is nilpotent. Finally, let $R^* := R\setminus\mathfrak{m}$ denote the group of invertible elements.

We denote by $\mathcal{P}(R)$ the set of $R^*$-orbits in $\mathcal{E}(R)$, i.e. $(r_1, r_2)\sim(ur_1, ur_2)$ for $u\in R^*$, and the quotient map by $\operatorname{pc}\colon \mathcal{E}(R)\rightarrow\mathcal{P}(R)$. Since $t\in R^*$ and since the $R^*$-action on $\mathcal{E}(R)$ obviously commutes with the $\mathrm{Bu}_3$-action in~(\ref{eq:action}), the latter reduces to a $\Gamma$-action on $\mathcal{P}(R)$. Then, the obvious quotient covering $\mathcal{C}(R)\rightarrow \mathcal{P}(R)$ is regular (Galois). Let $\Omega\subset\mathcal{P}(R)$ be an orbit, and let $\tilde{\Omega}\subset\mathcal{C}(R)$ be an orbit which maps to $\Omega$. It is clear that $\mathcal{H}(\tilde{\Omega})$ depends only on $\Omega$, i.e. independent of the choice of $\tilde{\Omega}$. Thus, we define $\mathcal{H}(\Omega)$ accordingly. Hence, for a proof of Theorem~\ref{thm:main}, it is enough to check the orbits $\Omega\subset\mathcal{P}(R)$ for rings $R$ and $\Omega\subset\mathcal{C}(W)$ for wheels $W$.

Let $\mathfrak{a}\subset\mathcal{P}(R)$ be a vertex or a region, and let $n(\mathfrak{a})$ be the ramification index of $\mathfrak{a}'\mapsto\mathfrak{a}$ for any $\mathfrak{a}'\subset\mathcal{C}(R)$ in the preimage of $\mathfrak{a}$. Clearly, $n(\mathfrak{a})$ is independent of the choice of $\mathfrak{a}'$. Then, we assign to each $\mathfrak{a}$ the weight of $w(\mathfrak{a}) = \frac{1}{n(\mathfrak{a})}$, and always consider $\mathcal{P}(R)$ with these weights. Thus, we redefine the Euler characteristic $\chi(\Omega)$ for an orbit $\Omega\subset\mathcal{P}(R)$ as the sum of weights over the vertices and the regions in $\Omega$ minus $|\Omega|$ (the number of edges). Let $\tilde{\Omega}\subset\mathcal{C}(R)$ be an orbit in the preimage of $\Omega$, and let $d$ be the degree of the covering $\tilde{\Omega}\rightarrow\Omega$. The following is a direct consequence of Lemma~\ref{lem:genus} and the definitions here:
\begin{equation*}
\chi(\tilde{\Omega}) = d\cdot\chi(\Omega).
\end{equation*}
Therefore, $\mathcal{H}(\Omega)$ is genus-zero if and only if $\chi(\Omega) > 0$. By a slight abuse of terminology, we say that an orbit $\Omega\subset\mathcal{P}(R)$ is \emph{genus-zero} if and only if $\chi(\Omega) > 0$. We now give a formula for $\chi(\Omega)$, which is straightforward. We say that a monovalent vertex is \emph{complete} if it has weight 1.
\begin{lemma}[Euler Characteristic Formula]
\label{lem:chi_P}
Let $\Omega\subset\mathcal{P}(R)$ be an orbit, let $\Omega_{YX}$ denote the set of regions in $\Omega$, and let $\Omega_\bullet$ and $\Omega_\circ$ denote the set of complete monovalent \emph{black} and \emph{white} vertices in $\Omega$, respectively. Then
\begin{equation*}
\chi(\Omega) = -  \frac{|\Omega|}{6} + \frac{2}{3}\cdot|\Omega_\bullet| +  \frac{1}{2}\cdot|\Omega_\circ| + \sum_{\mathfrak{a}\in\Omega_{YX}}w(\mathfrak{a}).
\end{equation*}
\end{lemma}

We now describe a standard way of choosing the pair of elements to denote an edge in $\mathcal{P}(R)$, although we do not restrict ourselves to this standard notation in the rest. Any edge is denoted by $\operatorname{pc}(r_1, r_2)$, where $(r_1, r_2)\in\mathcal{E}(R)$, i.e. $r_1, r_2 \in R$ and at least one of $r_1, r_2$ is in $R^*$. If $r_1\in R^*$, one has $\operatorname{pc}(r_1, r_2) = \operatorname{pc}(1, \frac{r_2}{r_1})$. If $r_1\in \mathfrak{m}$, then $r_2\in R^*$ and one has $\operatorname{pc}(r_1, r_2) = \operatorname{pc}(\frac{r_1}{r_2}, 1)$. Therefore, any edge can be denoted in the form of either $\operatorname{pc}(1, r)$ for some $r\in R$ or $\operatorname{pc}(m, 1)$ for some $m\in \mathfrak{m}$. It is clear that this form is unique for each edge. In particular, the number of edges in $\mathcal{P}(R)$ is given by $|R| + |\mathfrak{m}| = (|\Bbbk|+1)\cdot|\mathfrak{m}|$.

\begin{remark}
\label{rem:decompose}
Let $\mathfrak{m}^*$ denote the kernel of the group epimorphism $R^*\twoheadrightarrow \Bbbk^*$. Then, $\mathfrak{m}^*$ is a $p$-group, hence $R^*$ naturally splits as $R^* = \mathfrak{m}^* \oplus \Bbbk^*$. One can see this as follows: let $n$ be such that $\mathfrak{m}^{p^n} = 0$, then $(1+m)^{p^{(p^n+n)}} = 1$ for all $m\in \mathfrak{m}$. 
\end{remark}

\subsubsection{Rings and Wheels for $\mathfrak{m} = \langle p, t+1\rangle$}

In the case $\mathfrak{m} = \langle p, t+1\rangle$, we fix the following brief notation:
\begin{equation*}
\begin{aligned}
\lambda & := -1-t, \\
\omega_\ell & := (-t)^{p^\ell-1} + (- t)^{p^\ell-2} + \ldots = \sum_{i=0}^{p^\ell-1} (-t)^{i} \text{ for } \ell\geq 0, \\
\delta_\ell & := (-t)^{p^{\ell-1}\cdot(p-1)} + (- t)^{p^{\ell-1}\cdot(p-2)} + \ldots = \sum_{i=0}^{p-1} (-t)^{p^{\ell-1}\cdot i} \text{ for }\ell\geq 1.
\end{aligned}
\end{equation*}
Whenever we speak of a ring $R$, the notation above refers to elements of $R$; and whenever we speak of a wheel $W$, they refer to elements of $\Lambda$. Then, note that $\lambda\in \mathfrak{m}$ and $\delta_\ell\in \mathfrak{m}$. For a ring $R$, we use the following additional notation:
\begin{enumerate}
\item For any $a\in R$, let $\ell_0(a)$ denote the value for which
\begin{equation*}
\omega_\ell\cdot a(a-\lambda) = 0\ \text{if and only if}\ \ell\geq \ell_0(a).
\end{equation*}
This is well-defined because $\omega_\ell = \omega_{\ell-1}\delta_\ell$. Let $\ell_0$ denote the common value of $\ell_0(u)$ for any $u\in R^* = R\setminus \mathfrak{m}$. Note that $\omega_\ell = 0$ if and only if $\ell\geq \ell_0$.
\item For any $m\in \mathfrak{m}$, let $\ell_0'(m)$ denote the non-negative value for which
\begin{equation*}
(-1)^{p^{\ell_0(m)+\ell}}\cdot (1+\omega_{\ell_0(m)+\ell}\cdot(\lambda-m)) \in \langle t \rangle \subset R^*\ \text{if and only if}\ \ell\geq \ell_0'(m).
\end{equation*}
This is also well-defined since $(1+\omega_{\ell+1}(\lambda-m)) = (1+\omega_{\ell}(\lambda-m))^p$ for all $\ell\geq \ell_0(m)$.
\item In the case $p\neq 2$, the definition of $\ell_0'(m)$ can be simplified. Consider the decomposition $R^* = \mathfrak{m}^* \oplus \Bbbk^*$. Then $\langle t \rangle = \langle 1+\lambda \rangle \oplus \langle -1\rangle$. Since $(-1)^{p^{\ell_0(m)+\ell}} = -1$ and $(1+\omega_{\ell_0(m)+\ell}\cdot(\lambda-m))\in \mathfrak{m}^*$, one has
\begin{equation*}
(1+\omega_{\ell_0(m)+\ell}\cdot(\lambda-m))\in \langle 1+\lambda \rangle\subset\mathfrak{m}^*\ \text{if and only if}\ \ell\geq \ell_0'(m).
\end{equation*}
\end{enumerate}
Finally, note that one simply has $\Bbbk = \Lambda\slash \langle p, t+1\rangle = \mathbb{F}_p$.

\subsection{Restrictions on $\mathcal{C}(W)$ and $\mathcal{P}(R)$}

In this section, we establish formulae about the monovalent vertices and the regions in $\mathcal{C}(W)$ or $\mathcal{P}(R)$ for a wheel $W$ or a ring $R$. With these formulae, one can compute the Euler characteristic of an orbit by Lemma~\ref{lem:genus} or Lemma~\ref{lem:chi_P}.

\subsubsection{Monovalent Vertices}

\begin{lemma}
Let $W$ be a wheel. Then, there is no monovalent vertex in $\mathcal{C}(W)$.
\end{lemma}
\begin{proof}
There is an epimorphism $W\twoheadrightarrow\Bbbk^2$, hence an induced covering $\mathcal{C}(W)\rightarrow\mathcal{C}(\Bbbk^2)$. Thus, it is enough to show that $\mathcal{C}(\Bbbk^2)$ contains no monovalent vertex, i.e. no edge fixed by $X$ or $Y$. Just by comparing the first coordinates, we see that $(a_1,a_2)\in\mathcal{E}(\Bbbk^2)$ is not in the same $t$-orbit as $(a_2,-a_1-a_2)\in\mathcal{E}(\Bbbk^2)$, hence $\operatorname{c}(a_1,a_2)\neq \operatorname{c}(a_2,-a_1-a_2) = \operatorname{c}(a_1,a_2)\cdot X$ by~(\ref{eq:gamma-act}). Similarly by comparing the first coordinates, we see that no edge is fixed by $Y$.
\end{proof}

\begin{lemma}
\label{lem:r-vtx}
Let $R$ be a ring.
\begin{enumerate}
\item A complete monovalent black vertex in $\mathcal{P}(R)$ consists of an edge $\operatorname{pc}(1,r)$ where $r\in\langle t\rangle$ and $r^2+r+1 = 0$.
\item A complete monovalent white vertex in $\mathcal{P}(R)$ consists of an edge $\operatorname{pc}(1,r)$ where $-r\in\langle t\rangle$ and $r^2 = \frac{1}{t}$.
\end{enumerate}
Consequently, the number of complete monovalent black vertices is at most 3. For $\mathfrak{m} = \langle p, t+1\rangle$ with $p\neq 3$, the number is 0. The number of complete monovalent white vertices is at most 1. For $\mathfrak{m} = \langle p, t+1\rangle$ with $p\neq 2$, the number is 0.
\end{lemma}
\begin{proof}
The complete monovalent black vertices in $\mathcal{P}(R)$ are counted by the solutions of the equations $\operatorname{c}(1,r) = \operatorname{c}(1,r)\cdot X = \operatorname{c}(r,-r-1)$ and $\operatorname{c}(m,1) = \operatorname{c}(m,1)\cdot X = \operatorname{c}(1, -m-1)$. Clearly, the second equation has no solution $m\in \mathfrak{m}$, while the first equation is satisfied if and only if $r^2 + r + 1 = 0$ and $r\in\langle t\rangle \subset R^*$. The equality $r^2 + r + 1 = 0$ implies $r^3 = 1$, therefore there are at most 3 such vertices (those elements in the cyclic group $\langle t\rangle$ with order dividing 3). Moreover, if $p\neq 3$, the equality $r^2 + r + 1 = 0$ holds if and only if the order $\ord(r) = 3$. Therefore, the number of such vertices is 2 if $3\mid \ord(t)$ and 0 otherwise. But for $\mathfrak{m} = \langle p, t+1\rangle$, one has $\ord(t)\mid 2\cdot|\mathfrak{m}^*|$ which is not divisible by 3 for $p\neq 3$.

Similarly, the complete monovalent white vertices in $\mathcal{P}(R)$ are counted by the solutions of the equations $\operatorname{c}(1,r) = \operatorname{c}(1,r)\cdot Y = \operatorname{c}(-tr,-1)$ and $\operatorname{c}(m,1) = \operatorname{c}(m,1)\cdot Y = \operatorname{c}(-t, -m)$. As above, the second equation has no solution $m\in \mathfrak{m}$, while the first equation is satisfied if and only if $r^2 = \frac{1}{t}$ and $-r\in\langle t\rangle \subset R^*$. If $\ord(t)$ is even, $\langle t^2\rangle$ is properly contained in $\langle t\rangle$, hence there is no such vertex. If $\ord(t)$ is odd, there is a unique square root of $\frac{1}{t}$ in the cyclic group $\langle t\rangle$, hence there is 1 such vertex. Note that, for $\mathfrak{m} = \langle p, t+1\rangle$ with $p\neq 2$, one has $\ord(t)$ even.
\end{proof}

\subsubsection{Regions}
\label{sec:regions}

The lemmas in this section concern the regions in $\mathcal{C}(W)$ or $\mathcal{P}(R)$ in the case of $\mathfrak{m} = \langle p, t+1\rangle$ only. It is easier to establish corresponding results for $t+1\not\in\mathfrak{m}$, but we prefer not to include them here as they are not relevant.

\begin{lemma}
\label{lem:w-reg}
Let $W$ be a wheel with $\mathfrak{m} = \langle p, t+1\rangle$. Then, the size of any region in $\mathcal{C}(W)$ is a power of $p$.
\end{lemma}
\begin{proof}
Let $\operatorname{c}(a_1,a_2)\in\mathcal{C}(W)$. Then
\begin{equation*}
\operatorname{c}(a_1,a_2)\cdot (YX)^{p^\ell} = (-1)^{p^\ell}\cdot\operatorname{c}(a_1, (1+\omega_\ell\lambda)\cdot a_2 - \omega_\ell\cdot a_1).
\end{equation*}
For sufficiently large $\ell$, one has $\omega_\ell\cdot W = 0$, hence $\operatorname{c}(a_1, (1+\omega_\ell\lambda)\cdot a_2 - \omega_\ell\cdot a_1) = \operatorname{c}(a_1,a_2)$. Moreover, the factor of $(-1)^{p^\ell}$ can be ignored because of the following: for $p>2$, a certain power of $t$ acts on $W$ as $-1$; and for $p=2$, one has $(-1)^{p^\ell} = 1$. Thus, $\operatorname{c}(a_1,a_2)\cdot (YX)^{p^\ell} = \operatorname{c}(a_1,a_2)$ for sufficiently large $\ell$, which proves the statement.
\end{proof}

We now give a characterization of the orbits in the $\mathrm{Bu}_3$-set $\mathcal{E}(\Bbbk^2)$. There is a natural identification $\mathcal{E}(\Bbbk^2) = GL(2, \Bbbk)$ of underlying sets as follows: any element is $\mathcal{E}(\Bbbk^2)$ is a pair of linearly independent nonzero vectors in $\Bbbk^2$, hence it is identified with the matrix formed by putting the two vectors side by side as column vectors. This identification of sets allows a natural interpretation of the $\mathrm{Bu}_3$-action on $GL(2, \Bbbk)$; it is essentially matrix multiplication on the right. Here, in order to multiply a matrix in $\mathrm{Bu}_3$ with a matrix in $GL(2, \Bbbk)$, one first evaluates the former at $t=-1$ (since $t$ acts as $-1$ on $\Bbbk^2$), then reduces it modulo $p$. In other words, $\mathrm{Bu}_3$ acts on $GL(2, \Bbbk)$ via the composed epimorphism $\mathrm{Bu}_3 \twoheadrightarrow SL(2, \mathbb{Z}) \twoheadrightarrow SL(2, \Bbbk)$. Therefore, the orbits in $\mathcal{E}(\Bbbk^2) = GL(2,\Bbbk)$ are the cosets of $SL(2,\Bbbk)$, i.e. they are characterized by the value of the determinant.

We now define a particular surjective function $\operatorname{r}\colon \mathcal{C}(\Bbbk^2) \rightarrow \mathcal{C}(\Bbbk)$, which is useful in describing the regions in $\mathcal{C}(\Bbbk^2)$: for any $\operatorname{c}(v_1, v_2) \in \mathcal{C}(\Bbbk^2)$, let $\operatorname{r}(\operatorname{c}(v_1, v_2)) := \operatorname{c}(v_1)$. Here, $\operatorname{c}(v_1)$ is meaningful when we treat $v_1$ as a pair of elements of $\Bbbk$. Note that $\operatorname{r}$ is not $\Gamma$-equivariant, it is simply a function between the underlying sets. 

\begin{lemma}
\label{lem:k2-reg}
The size of any region in $\mathcal{C}(\Bbbk^2)$ is equal to $p$. Two edges $c_1, c_2\in \mathcal{C}(\Bbbk^2)$ are in the same region if and only if they are in the same orbit and $\operatorname{r}(c_1) = \operatorname{r}(c_2)$. The function $\operatorname{r}$ remains surjective when it is restricted to any orbit in $\mathcal{C}(\Bbbk^2)$.
\end{lemma}
\begin{proof}
For the first statement, first observe that $\omega_1$ annihilates the wheel $\Bbbk^2$, therefore the size of any region is at most $p$. Secondly, $t$ acts on $\Bbbk^2$ as $-1$, hence $\operatorname{c}(a_1,a_2)\cdot YX = \operatorname{c}(-a_1, a_1 - a_2) = \operatorname{c}(a_1, a_2 - a_1) \neq \operatorname{c}(a_1, a_2)$, therefore the size of any region is greater than 1 (at least $p$). The inequality here can be shown by comparing the second coordinates and noting that they are not in the same $t$-orbit.

Two edges $c_1,c_2$ in the same region are clearly in the same orbit, and the equality $\operatorname{r}(c_1) = \operatorname{r}(c_2)$ can be seen by noting $\operatorname{c}(a_1,a_2)\cdot YX = \operatorname{c}(a_1, a_2 - a_1)$. The other statements are immediate consequences of simple facts of linear algebra once we have the above characterization of the orbits in $\mathcal{E}(\Bbbk^2)$. The function $\operatorname{r}$ is surjective when restricted to any orbit, because keeping the first column of a matrix in $GL(2,\Bbbk)$ fixed, one can arrange the second column to obtain an arbitrary value of the determinant. Similarly, $\operatorname{r}(c_1) = \operatorname{r}(c_2)$ implies that $c_1$ and $c_2$ are in the same region provided that they are in the same orbit, becuase keeping the first column and the determinant of a matrix in $GL(2,\Bbbk)$ fixed, one can vary the second column only by adding the multiples of the first column.
\end{proof}

\begin{lemma}
\label{lem:r-reg}
Let $R$ be a ring with $\mathfrak{m} = \langle p, t+1\rangle$. The sizes and the weights of the regions in $\mathcal{P}(R)$ are as follows:
\begin{enumerate}
\item\label{lem:r-reg:bulk} For any $r\in R$, the size of the region which contains $\operatorname{pc}(1,r)\in\mathcal{P}(R)$ is $p^{\ell_0}$ and the weight is 1 (the region is unramified).
\item\label{lem:r-reg:small} For any $m\in \mathfrak{m}$, the size of the region which contains $\operatorname{pc}(m,1)\in\mathcal{P}(R)$ is $p^{\ell_0(m)}$ and the weight is $p^{-\ell_0'(m)}$.
\end{enumerate}
\end{lemma}
\begin{proof}
As in the proof of Lemma~\ref{lem:w-reg}, one has
\begin{equation*}
\operatorname{pc}(1,r)\cdot(YX)^{p^\ell} = (-1)^{p^\ell}\cdot\operatorname{pc}(1, (1+\omega_\ell\lambda)\cdot r - \omega_\ell) = \operatorname{pc}(1, r+\omega_\ell(\lambda r - 1)).
\end{equation*}
Clearly, $\operatorname{pc}(1, r+\omega_\ell(\lambda r - 1)) = \operatorname{pc}(1,r)$ if and only if $\omega_\ell(\lambda r - 1) = 0$. The latter is equivalent to $\omega_\ell = 0$ since $(\lambda r - 1)\in R^*$, which holds if and only if $\ell \geq \ell_0$. Moreover, the equality $\operatorname{c}(1,r)\cdot(YX)^{p^{\ell_0}} = \operatorname{c}(1,r)$ similarly holds. This finishes the proof of item~\ref{lem:r-reg:bulk}. Similarly, 
\begin{equation*}
\operatorname{pc}(m,1)\cdot(YX)^{p^\ell} = (-1)^{p^\ell}\cdot\operatorname{pc}(m, 1+\omega_\ell\lambda - \omega_\ell m) = \operatorname{pc}(m, 1+\omega_\ell(\lambda - m)).
\end{equation*}
Now, $\operatorname{pc}(m, 1+\omega_\ell(\lambda - m)) = \operatorname{pc}(m,1)$ if and only if $(1+\omega_\ell(\lambda - m))\cdot m = m$, that is, $\omega_\ell \cdot m(m-\lambda) = 0$, which holds precisely for $\ell\geq \ell_0(m)$. Then,
\begin{equation*}
\operatorname{c}(m,1)\cdot(YX)^{p^{\ell_0(m)+\ell}} = (-1)^{p^{\ell_0(m)+\ell}}\cdot\operatorname{c}(m, 1+\omega_{\ell_0(m)+\ell}\cdot(\lambda - m)).
\end{equation*}
But, $(-1)^{p^{\ell_0(m)+\ell}}\cdot\operatorname{c}(m, 1+\omega_{\ell_0(m)+\ell}\cdot(\lambda - m)) = \operatorname{c}(m,1)$ if and only if $(-1)^{p^{\ell_0(m)+\ell}}\cdot(1+\omega_{\ell_0(m)+\ell}\cdot(\lambda - m))$ is in $\langle t \rangle \subset R^*$, which holds precisely for $\ell\geq \ell_0'(m)$. This finishes the proof of item~\ref{lem:r-reg:small}.
\end{proof}

\section{The Proof of the Main Theorem}
\label{sec:proof}

In this section, we prove Theorem~\ref{thm:main} by checking the orbits $\Omega\subset\mathcal{C}(W)$ for wheels $W$ and the orbits $\Omega\subset\mathcal{P}(R)$ for wheels $R$, case by case. In each subcase, we write one main statement in italics at the beginning. If the statement is of the \enquote{infinite} kind, e.g. one which claims that there is no genus-zero orbit $\Omega$ in $\mathcal{P}(R)$ if $R$ is in a certain infinite class of rings which may be characterized by a certain condition, then we explicitly prove the statement. However, if the statement is of the \enquote{finite} kind, e.g. about properties of a particular orbit $\Omega\subset\mathcal{C}(W)$ for a particular wheel $W$, we leave its proof to the reader. During the proof, we use all of the lemmas in the previous sections implicitly.

We begin with an observation: For any prime $p$, there is only one orbit in $\mathcal{P}(\Bbbk)$. Because there are only two regions in $\mathcal{P}(\Bbbk)$, one of which has size $p$ and the other has size 1; then, it is only left to show that the edge in the region of size 1, namely $\operatorname{pc}(0,1)$, is not fixed by $X\in\Gamma$.

\paragraph{\textbf{The case $p > 7$:}}
\emph{There is no genus-zero orbit $\Omega$}. Since any ring or wheel admits an epimorphism onto $\Bbbk$, it is enough to show $\chi(\mathcal{P}(\Bbbk))\leq 0$. There is no complete monovalent vertex in $\mathcal{P}(\Bbbk)$, hence $\chi(\mathcal{P}(\Bbbk)) = -\frac{p+1}{6} + 2 \leq 0$. \qed

\paragraph{\textbf{The case $p = 7$:}}
\emph{There is no genus-zero orbit $\Omega$, except in $\mathcal{P}(R)$ for $R = \Bbbk$ or $R = \Bbbk[\lambda]\slash \lambda^2$}. First, consider wheels. Since any wheel admits an epimorphism onto $\Bbbk^2$, it is enough to observe that all orbits in $\mathcal{C}(\Bbbk^2)$ are of positive genus, because all regions are of size 7. Now, consider rings. Note that any ring other than the aforementioned ones admits an epimorphism onto at least one of these: $\Bbbk[\lambda]\slash \lambda^3$, $\mathbb{Z}_{49}[\lambda]\slash (\lambda-7k)$ and $\mathbb{Z}_{49}[\lambda]\slash\langle 7\lambda, \lambda^2-7k \rangle$ for some $k = 0, 1, \ldots, 6$. Then, one simply checks that there is no genus-zero orbit in $\mathcal{P}(\Bbbk[\lambda]\slash \lambda^3)$, $\mathcal{P}(\mathbb{Z}_{49}[\lambda]\slash (\lambda-7k))$ or $\mathcal{P}(\mathbb{Z}_{49}[\lambda]\slash\langle 7\lambda, \lambda^2-7k \rangle)$ for any value of $k$. As for the exceptional rings, there are three orbits in $\mathcal{P}(\Bbbk[\lambda]\slash \lambda^2)$, each of which is genus-zero. As $\Omega$ ranges these orbits, $\mathcal{H}(\Omega)$ ranges $H_1(7,1)$, $H_2(7,1)$, and $H_3(7,1)$. Finally, $\mathcal{H}(\mathcal{P}(\Bbbk)) = H(7,0)$. \qed

\paragraph{\textbf{Wheels with $p = 5$:}}
\emph{For a wheel $W$ not annihilated by $\omega_1$, there is no genus-zero orbit in $\mathcal{C}(W)$.} Let $W$ be such a wheel and let $\Omega$ be an orbit in $\mathcal{C}(W)$; we will show that the number of regions in $\Omega$ is less than $\frac{1}{6}\cdot|\Omega|$, which lets one conclude $g(\Omega) > 0$ since there is no monovalent vertex in $\Omega$. Note that the size of any region is 5 or at least 25. If an edge $\operatorname{c}(a_1,a_2)$ is contained in a region of size 5, then $\operatorname{c}(a_1,a_2) = \operatorname{c}(a_1,a_2+\omega_1(\lambda\cdot a_2 - a_1))$. This implies either $\omega_1(a_1 - \lambda\cdot a_2) = 0$ or $((1+\lambda)^{k}-1)\cdot a_1 = 0$ for some $k$ for which $(1+\lambda)^{k}-1$ does not annihilate $W$. Let $n$ be greatest such that $\omega_n\lambda = (1+\lambda)^{5^n}-1$ does not annihilate $W$. The equation $((1+\lambda)^{k}-1)\cdot a_1 = 0$ implies $\omega_n\lambda\cdot a_1 = 0$. Consider the images of the submodules $\{a\in W\mid \omega_1\cdot a= 0\}$ and $\{a\in W\mid \omega_n\lambda\cdot a= 0\}$ under an epimorphism $W\twoheadrightarrow\Bbbk^2$. These images are subspaces of dimension at most 1 since these equations do not identically hold in $W$, and $W$ is generated by any two elements which project to linearly independent vectors in $\Bbbk^2$. Therefore, if an edge $\operatorname{c}(a_1,a_2)$ is contained in a region of size 5, then $a_1$ projects into one of these two 1-dimensional subspaces. Overall, out of the 24 nonzero vectors in $\Bbbk^2$, at most 8 of them (union of two distinct 1-dimensional subspaces) can be equal to the projection of $a_1$. Let $\Omega'$ be the image of $\Omega$ under the covering $\mathcal{C}(W) \rightarrow \mathcal{C}(\Bbbk^2)$. The restriction on $a_1$ is equivalently expressed as follows: at most 4 out of the 12 regions in $\Omega'$ can be the image of the region which contains $\operatorname{c}(a_1,a_2)$. This shows that at least two thirds of the edges in $\Omega$ are contained in regions of size at least 25. Hence, the number of regions is bounded by $\frac{1}{5}\cdot\frac{1}{3}\cdot|\Omega| + \frac{1}{25}\cdot\frac{2}{3}\cdot|\Omega| < \frac{1}{6}\cdot|\Omega|$.\qed

\medskip

\emph{For a wheel $W$ annihilated by $\omega_1$ and $\Omega$ an orbit in $\mathcal{C}(W)$, one has $\mathcal{H}(\Omega) = \tilde{H}(5)$ if $W = \Bbbk^2$, and $\mathcal{H}(\Omega) = \tilde{I}(5)$ otherwise.} First, observe that all regions in $\mathcal{C}(W)$ are of size 5, hence any orbit is of genus zero and is isomorphic to its image in $\mathcal{C}(\Bbbk^2)$. Moreover, $\mathbb{L} := (\mathbb{Y}\mathbb{X})^5$ acts as $t^5$ on all pairs in $\mathcal{E}(W)$ since $\mathbb{L}\equiv -1\equiv t^5\cdot 1 \pmod{\omega_1}$. Hence, $\mathcal{H}(\Omega)$ is uniquely determined once the order of the $t$-action on $W$, i.e. the depth $\operatorname{d}(\mathcal{H}(\Omega))$, is known. Since $t^5\equiv -1\pmod{\omega_1}$, the order is either 2 or 10; in fact, it is 2 if and only if $W = \Bbbk^2$.\qed

\paragraph{\textbf{Rings with $p = 5$:}}
\emph{For a ring $R$ in which $\omega_1\neq 0$, there is no genus-zero orbit in $\mathcal{P}(R)$ unless $R = \mathbb{Z}_{25}[\lambda]\slash (\lambda-5k)$ for some $k = 1,2,3,4$.} First suppose that $\lambda\not\in R\omega_1$. By replacing $R$ with $R\slash\langle 5\omega_1, \omega_1\lambda \rangle$ if necessary, we assume that $5\omega_1 = \omega_1\lambda = 0\in R$, hence $\omega_2 = 0$. A region in $\mathcal{P}(R)$ which is of size 1 and weight 1 must consist of an edge $\operatorname{pc}(m, 1)$ where $1+\lambda-m \in \langle 1+\lambda \rangle$ and $m(m-\lambda) = 0$. The first condition alone shows that all such edges are distinct modulo $\omega_1$, because $\langle 1+\lambda \rangle = \{1, 1+\lambda, \ldots, (1+\lambda)^4\}$ and $\lambda\not\equiv 0\pmod{\omega_1}$. Now, let $\Omega$ be an orbit in $\mathcal{P}(R)$ and $\Omega'$ be its image under the 5-fold covering $\mathcal{P}(R) \rightarrow \mathcal{P}(R\slash\omega_1)$. If $\Omega$ contains regions of size 1 and weight 1, they all project to distinct regions in $\Omega'$, hence there are at most $d$ such regions where $d$ is the degree of the covering $\Omega' \rightarrow \mathcal{P}(\Bbbk)$. This is because all such regions must project to the unique region of size 1 in $\mathcal{P}(\Bbbk)$. On the other hand, note that there are exactly $d$ regions of size 5 in $\Omega'$; namely, those which project to the unique region of size 5 in $\mathcal{P}(\Bbbk)$. The corresponding regions in $\Omega$ are of size 25, which implies that the covering $\Omega\rightarrow\Omega'$ is 5-fold and that there are exactly $d$ regions of size 25. Finally, the degree of the covering $\Omega \rightarrow \mathcal{P}(\Bbbk)$ is $5d$, hence $|\Omega| = 30d$. As a consequence of all of this, the sum of weights over the regions in $\Omega$ is bounded above by $d + \frac{4d}{5} + d < \frac{|\Omega|}{6}$. Since there is no complete monovalent vertex in $\mathcal{P}(R)$, one concludes $\chi(\Omega) < 0$.

Now suppose that $\lambda = \omega_1r$ for some $r\in R$. Note that $\omega_1 = 5 + \lambda\theta$ for $\theta\in\mathfrak{m}$, hence $\lambda = 5r(1-\theta r)^{-1}$. This requires that $R = \mathbb{Z}_{5^n}$ for some $n\geq 2$ and $\lambda = 5k$. Then, it is only left to check that there is only one orbit in $\mathcal{P}(R)$ whose Euler characteristic is negative when $n=3$ with any value of $\lambda$ or when $n=2$ with $\lambda=0$. When $n = 2$ and $\lambda = 5k$ for some $k = 1,2,3,4$, one has $\chi(\mathcal{P}(R)) = 1$ and $\mathcal{H}(\mathcal{P}(R)) = H(25; a)$ where $a = -5k-1$.\qed

\medskip

\emph{For a ring $R$ in which $\omega_1 = 0$, except the cases $R = \Bbbk$ and $R = \Bbbk[\lambda]\slash \lambda^2$, there is one orbit $\Omega_0$ in $\mathcal{P}(R)$ with $\mathcal{H}(\Omega_0) = I(5,1)$ and one has $\mathcal{H}(\Omega) = \tilde{I}(5)$ for any other orbit $\Omega$.} First, note that $\lambda^2\neq 0$, otherwise $R$ must be one of the two exceptional rings. Hence, there are only two regions in $\mathcal{P}(R)$ with size 1 and weight 1; namely, $\operatorname{pc}(0, 1)$ and $\operatorname{pc}(\lambda, 1)$. Because, these are the only values $m$ which satisfy $1+\lambda-m \in \langle 1+\lambda \rangle = \{1, 1+\lambda, \ldots, (1+\lambda)^4\}$ and $m(m-\lambda) = 0$. Moreover, these two edges are in the same orbit $\Omega_0$. Then, it is easy to deduce that $I(5,1) := \mathcal{H}(\Omega_0)$ is uniquely determined, by arguments very similar to those in the case of the wheels. Any other orbit $\Omega$ consists of regions which are either of size 5 or of size 1 but weight $\frac{1}{5}$. Hence, in the preimage of $\Omega$ in $\mathcal{C}(R)$, all regions are of size 5. Then, by similar arguments again, one deduces that $\mathcal{H}(\Omega) = \tilde{I}(5)$, which was already defined. As for the exceptional rings, there are three orbits in $\mathcal{P}(\Bbbk[\lambda]\slash \lambda^2)$ each of which is genus-zero. As $\Omega$ ranges these orbits, $\mathcal{H}(\Omega)$ ranges $I(5,1)$, $H_1(5,1)$, and $H_2(5,1)$. Finally, $\mathcal{H}(\mathcal{P}(\Bbbk)) = H(5,0)$. \qed

\paragraph{\textbf{Wheels with $p = 3$:}}

\textit{For a wheel $W$ not annihilated by $\omega_1$, there is no orbit of genus zero in $\mathcal{C}(W)$ unless $W = \mathbb{Z}_9\cdot e_1 \oplus \mathbb{Z}_3\cdot e_2$ with $\lambda\cdot e_1 = 0$ and $\lambda\cdot e_2 = 3e_1$.} The proof is similar to the case of $p=5$. Let $\Omega$ be an orbit in $\mathcal{C}(W)$ and $\Omega'$ be its image in $\mathcal{C}(\Bbbk^2)$. The size of any region in $\Omega$ is 3 or at least 9, and only 2 out of the 4 regions in $\Omega'$ can be the image of a region of size 3 in $\Omega$. This comes from considering the images of the submodules $\{a\in W\mid \omega_1\cdot a= 0\}$ and $\{a\in W\mid \omega_n\lambda\cdot a= 0\}$ in $\Bbbk^2$, as in the case $p=5$. First suppose that at most 1 out of the 4 regions in $\Omega'$ is the image of a region of size 3 in $\Omega$. Then, the number of regions in $\Omega$ is less than or equal to $\frac{1}{3}\cdot\frac{1}{4}\cdot|\Omega| + \frac{1}{9}\cdot\frac{3}{4}\cdot|\Omega| = \frac{1}{6}\cdot|\Omega|$, hence $g(\Omega) > 0$. Now suppose that exactly 2 out of the 4 regions are as such. In other words, $\Omega$ contains edges $\operatorname{c}(a_1, a_2)$ and $\operatorname{c}(b_1, b_2)$ such that both of these edges lie in regions of size 3 and $a_1$ and $b_1$ project to linearly independent vectors in $\Bbbk^2$. Then, the images of the submodules $\{a\in W\mid \omega_1\cdot a= 0\}$ and $\{a\in W\mid \omega_n\lambda\cdot a= 0\}$ must be distinct 1-dimensional subspaces, such that $a_1$ and $b_1$ project into these. Thus one must have $n = 0$, otherwise the latter subspace contains the former. Now, w.l.o.g, let $a_1$ and $b_1$ project into the images of $\{a\in W\mid \omega_1\cdot a= 0\}$ and $\{a\in W\mid \lambda\cdot a= 0\}$ respectively. Thus, $\omega_1(a_1 - \lambda\cdot a_2) = 0$ and $((1+\lambda)^{k}-1)\cdot b_1 = 0$, $((1+\lambda)^{k}-1)\cdot b_2 = \omega_1(\lambda\cdot b_2 - b_1)$ for some $k$ coprime to 3. Note that $(1+\lambda)^{k}-1 = \lambda s$ for $s\in \Lambda\setminus\mathfrak{m}$. Then, $\lambda\cdot b_1 = 0$ and $\omega_1\cdot b_1 = (\omega_1 - s)\lambda\cdot b_2$. The last equation implies, in particular, that $\lambda^2$ annihilates $W$ because $\lambda^2\cdot b_2 = (\omega_1 - s)^{-1}\cdot\omega_1\lambda\cdot b_1 = 0$. Here, $(\omega_1 - s)^{-1}$ is meaningful because even though $(\omega_1 - s)$ may not be invertible in $\Lambda$, its action on $W$ is invertible. This allows one to replace $\omega_1 = 3 + 3\lambda + \lambda^2$ by $3(1+\lambda)$ in the equations. Hence, by introducing the brief notation $a_1' = (1+\lambda)\cdot(a_1 - \lambda\cdot a_2)$ and $b_2' = (1+\lambda)^{-1}\cdot(\omega_1-s)\cdot b_2$, we re-express the equations as follows: $3a_1' = 0$, $\lambda\cdot b_1 = 0$ and $3b_1 = \lambda\cdot b_2'$. Since $a_1'$ and $b_1$ generate $W$, one can write $b_2' = \phi_1\cdot a_1' + \phi_2\cdot b_1$ for some $\phi_1, \phi_2\in\Lambda$; moreover, $\phi_1\not\in \mathfrak{m}$ since $b_1$ and $b_2'$ project to linearly independent vectors in $\Bbbk^2$. Thus, $3b_1 = \lambda\cdot b_2' = \lambda\phi_1\cdot a_1'$. Finally, let $e_1$ denote $b_1$ and $e_2$ denote $\phi_1\cdot a_1'$, then $e_1$ and $e_2$ generate $W$ and the equations take the form $3e_2 = 0, \lambda\cdot e_1 = 0, 3e_1 = \lambda\cdot e_2$; which shows that $W$ is the exceptional wheel introduced in the beginning. For this exceptional wheel $W$, there are two orbits in $\mathcal{C}(W)$. As $\Omega$ ranges these orbits, $\mathcal{H}(\Omega)$ ranges $\tilde{H}_1(9)$ and $\tilde{H}_2(9)$. \qed

\medskip

For a wheel $W$ annihilated by $\omega_1$ and $\Omega$ an orbit in $\mathcal{C}(W)$, one has $\mathcal{H}(\Omega) = \tilde{H}(3)$ if $W = \Bbbk^2$, and $\mathcal{H}(\Omega) = \tilde{I}(3)$ otherwise. This is proven in a way completely analogous to the case $p = 5$. \qed

\paragraph{\textbf{Rings with $p = 3$:}} \textit{For a ring $R$ in which $\omega_2\neq 0$, there is no genus-zero orbit in $\mathcal{P}(R)$ unless $R = \mathbb{Z}_{27}[\lambda]\slash (\lambda-3k)$ for some $k = 1,2,4,5,7,8$.} First suppose that $\omega_1\lambda\not\in R\omega_2$. By replacing $R$ with $R\slash\langle 3\omega_2, \omega_2\lambda \rangle$ if necessary, we assume that $3\omega_2 = \omega_2\lambda = 0\in R$, hence $\omega_3 = 0$. If $\mathcal{P}(R)$ contains a complete monovalent vertex, there is $r\in\langle t\rangle$ such that $r^2+r+1 = 0$. The candidates for this equation are $\{1, 1+\omega_1\lambda, (1+\omega_1\lambda)^2\}$, but $(1+\omega_1\lambda)^2 + (1+\omega_1\lambda) + 1 = \delta_2 \neq 0$. Hence, one necessarily has $3 = 1 + 1 + 1 = 0\in R$. In this case, $\omega_2 = \lambda^8$, thus $R = \Bbbk[\lambda]\slash \lambda^9$; then, there is no genus-zero orbit in $\mathcal{P}(R)$. Henceforth, we assume that $\mathcal{P}(R)$ contains no complete monovalent vertex. Let $\Omega$ be an orbit in $\mathcal{P}(R)$, let $\Omega'$ be its image in $\mathcal{P}(R\slash\omega_2)$ and let $d$ be the degree of the covering $\Omega'\rightarrow \mathcal{P}(\Bbbk)$. As in the case $p=5$, $\Omega$ contains at most $d$ regions of size 1 and weight 1, it contains exactly $\frac{d}{3}$ regions of size 27 and the covering $\Omega\rightarrow\Omega'$ is 3-fold. Consequently, the sum of weights over the regions in $\Omega$ is less than or equal to $d + \frac{2d}{3} + \frac{d}{3} = \frac{|\Omega|}{6}$, hence $\chi(\Omega)\leq 0$.

Now suppose that $\omega_1\lambda = \omega_2r$ for some $r\in R$, but $\lambda\not\in R\omega_1$. Note that $\delta_2 = 3 + \lambda\theta$ for $\theta\in\mathfrak{m}$, hence $\omega_1\lambda = \omega_1\delta_2r$ implies $\omega_1\lambda = 3\omega_1r(1-\theta r)^{-1}$. This implies that $\omega_1\lambda = 3k\omega_1$ for some integer $k$.
Therefore, $\omega_2 = \omega_1\delta_2 = 3\omega_1(1+k\theta)$. We may assume $3\omega_2 = 0\in R$, hence $9\omega_1 = 0$ and $\omega_3 = 0$. As in the previous paragraph, there is no complete monovalent vertex in $\mathcal{P}(R)$, since $3\neq 0$, $\delta_2\neq 0$ and $\omega_1\neq 0$. Now, let $\Omega$ be an orbit in $\mathcal{P}(R)$, let $\Omega'$ be its image in $\mathcal{P}(R\slash\omega_1)$ and let $d$ be the degree of the covering $\Omega'\rightarrow \mathcal{P}(\Bbbk)$. Since $\lambda\not\in R\omega_1$, $\Omega$ contains at most $3d$ regions of size 1 and weight 1, it contains exactly $d$ regions of size 27 and the covering $\Omega\rightarrow\Omega'$ is 9-fold. Therefore, the sum of weights over the regions in $\Omega$ is less than or equal to $3d + \frac{6d}{3} + d = \frac{|\Omega|}{6}$, hence $\chi(\Omega)\leq 0$. Finally, suppose that $\lambda\in R\omega_1$. As in the case $p =5$, this requires that $R=\mathbb{Z}_{3^n}$ for some $n\geq 3$ and $\lambda = 3k$. Then, it is only left to check that there is one orbit in $\mathcal{P}(R)$ whose Euler characteristic is negative when $n=4$ with any value of $\lambda$ or when $n=3$ with $\lambda=9k$. When $n=3$ and $\lambda = 3k$ for some $k = 1,2,4,5,7,8$, one has $\chi(\mathcal{P}(R)) = 2$ and $\mathcal{H}(\mathcal{P}(R)) = H(27; a)$ where $a=-3k-1$.\qed

\medskip

\textit{For a ring $R$ in which $\omega_2 = 0$ but $\omega_1\lambda^2 \neq 0$, there is no genus-zero orbit in $\mathcal{P}(R)$.} Let $\Omega$ be an orbit in $\mathcal{P}(R)$. First suppose that $\Omega$ contains no region of size 1 and weight 1 and it contains no monovalent vertex. Let $d$ be the degree of the covering $\Omega\rightarrow\mathcal{P}(\Bbbk)$. Then, $\Omega$ contains exactly $\frac{d}{3}$ regions of size 9, hence the sum of weights over the regions in $\Omega$ is less than or equal to $\frac{d}{3} + \frac{d}{3} = \frac{|\Omega|}{6}$, thus $\chi(\Omega)\leq 0$. Henceforth, we assume that $\Omega$ contains either a region of size 1 and weight 1 or a complete monovalent vertex. Note that $\omega_1\lambda^2\not\in R\cdot 3\omega_1$. Otherwise, one has $\omega_1\lambda^2 = 3\omega_1r$ for some $r\in R$, then $\delta_2 = 3(1+ \omega_1\lambda + \omega_1^2r)$, hence $\omega_1\lambda^2 = 3\omega_1r = \delta_2\omega_1(1+ \omega_1\lambda + \omega_1^2r)^{-1}r = 0$, since $\delta_2\omega_1 = \omega_2 = 0$. Thus, by replacing $R$ with the appropriate quotient if necessary, we assume that $3\omega_1 = \omega_1\lambda^3 = 0\in R$. Then, note that $\lambda^2\not\in R\cdot 3$, hence there is an epimorphism $R \twoheadrightarrow \Bbbk[\lambda]\slash\lambda^3$. Now, let $\Omega'$ be the image of $\Omega$ in $\mathcal{P}(\Bbbk[\lambda]\slash\lambda^3)$. There are three orbits in $\mathcal{P}(\Bbbk[\lambda]\slash\lambda^3)$ two of which contain regions of size 1 and weight 1, but no monovalent vertex; while the other one contains complete monovalent vertices, but no region of size 1. We cover the two cases separately. Let $d$ be the degree of the covering $\Omega\rightarrow\Omega'$.

For the former case, first observe that there are only two regions in $\mathcal{P}(R)$ with size 1 and weight 1; namely, $\operatorname{pc}(0, 1)$ and $\operatorname{pc}(\lambda, 1)$. Because, these are the only values $m$ which satisfy $1+\lambda-m \in \langle 1+\lambda \rangle = \{1, 1+\lambda, \ldots, (1+\lambda)^8\}$ and $m(m-\lambda) = 0$. These two edges project into distinct orbits in $\mathcal{P}(\Bbbk[\lambda]\slash\lambda^3)$, hence $\Omega$ contains only one of them. Assume that $\operatorname{pc}(0, 1)\in\Omega$; we will not treat the other case since it is completely analogous. Then, $\Omega'$ contains two more regions of size 1, but these are of weight $\frac{1}{3}$. It is easy to verify that the preimage of any of these two regions under the covering $\Omega\rightarrow\Omega'$ contains a region of size greater than 1, hence $d>1$. We will now show that this preimage contains no region of size 1 and weight $\frac{1}{3}$ or of size 3 and weight 1. For this, it is enough to observe that $(1 + \omega_1(\lambda - m)) \not\in\langle 1+\lambda\rangle$ for any $m$ which projects to $\pm \lambda^2\in \Bbbk[\lambda]\slash\lambda^3$. Indeed, for such $m$, one has $1 + \omega_1(\lambda - m) = 1 + \omega_1\lambda + (\pm\omega_1\lambda^2) \not\in\langle 1+\lambda\rangle$. Using the observations above, one can bound the sum of weights over the regions of $\Omega$ by $\frac{d+2}{3} + \frac{2d}{9} + d < 2d = \frac{|\Omega|}{6}$, hence $\chi(\Omega)\leq 0$. In the latter case, all regions in $\Omega$ are of size 9. Because, if $m$ projects to $-\lambda\in\Bbbk[\lambda]\slash\lambda^2$, then $m(m-\lambda)$ projects to $-\lambda^2\in\Bbbk[\lambda]\slash\lambda^3$, hence $\omega_1\cdot m(m-\lambda)\neq 0$. In particular, one has $d\geq 3$, hence $|\Omega|\geq 36$. Then, since there are at most 3 complete monovalent vertices in $\Omega$, one deduces $\chi(\Omega)\leq -\frac{|\Omega|}{6} + \frac{|\Omega|}{9} + 2 \leq 0$.\qed

\medskip

\textit{For a ring $R$ in which $\omega_2 = 0$ and $\omega_1\lambda^2 = 0$, but $\omega_1\neq 0$, except the cases $R = \Bbbk[\lambda]\slash\lambda^3$, $R = \Bbbk[\lambda]\slash\lambda^4$, $R = \mathbb{Z}_9[\lambda]\slash\lambda^2$, $R = \mathbb{Z}_9[\lambda]\slash\langle3\lambda, \lambda^2\rangle$, $R = \mathbb{Z}_9[\lambda]\slash(\lambda - 3k)$ for some $k = 0,1,2$, there are exactly two genus-zero orbits in $\mathcal{P}(R)$.} As $\Omega$ ranges these orbits, $\mathcal{H}(\Omega)$ ranges $I_1(9,1)$ and $I_2(9,1)$ if $\omega_1\lambda = 0$, and it ranges $I_1(9,2)$ and $I_2(9,2)$ otherwise. For the proof, first note that $3\neq 0$ and $\lambda^2\neq 0$; otherwise $R$ must be one of the exceptional rings. Moreover, $\delta_2 = 3 + 3\omega_1\lambda + \omega_1^2\lambda^2 = 3(1+\omega_1\lambda)$, hence $3\omega_1 = \delta_2\omega_1(1+\omega_1\lambda)^{-1} = 0$. This, in turn, implies $\delta_2 = 3 + 3\omega_1\lambda = 3$. Therefore, there is no complete monovalent vertex in $\mathcal{P}(R)$ since $3\neq 0$, $\delta_2 = 3 \neq 0$ and $\omega_1\neq 0$. Thus, if an orbit $\Omega$ in $\mathcal{P}(R)$ contains no region of size 1 and weight 1, one has $\chi(\Omega)\leq 0$ as above. Hence, let $\Omega$ contain such a region $\operatorname{pc}(m, 1)$. Let $\Omega'$ be the image of $\Omega$ and $\operatorname{pc}(m', 1)$ be the image of $\operatorname{pc}(m, 1)$ in $\mathcal{P}(R\slash\omega_1\lambda)$. Since $\operatorname{pc}(m', 1)$ is also a region of size 1 and weight 1, one has $1+\lambda-m'\in\{1, 1+\lambda, (1+\lambda)^2\}$, i.e. $m'\in\{0, \lambda, -\lambda(1+\lambda)\}$. However, if $m' = -\lambda(1+\lambda)$, then $m(m-\lambda)\neq 0$ because $\lambda^2\neq 0$, hence $\Omega'$ contains $\operatorname{pc}(0, 1)$ or $\operatorname{pc}(\lambda, 1)$. Now, observe that $\lambda \not\in R\cdot 3 + R\cdot \omega_1\lambda$, hence there is an epimorphism $R\slash\omega_1\lambda \twoheadrightarrow \Bbbk[\lambda]\slash\lambda^2$. (Suppose for the contrary that $\lambda = 3r_1 + \omega_1\lambda r_2$. Then, $\lambda = 3r_1(1-\omega_1r_2)^{-1}$, which implies $\lambda = 3k$ for an integer $k$. Then, $\omega_1 = 3(1+3k+3k^2)$, hence $9=0$ because $3\omega_1 = 0$. Thus, $R$ is one of the exceptional rings.) Now, consider the induced covering $\mathcal{P}(R\slash\omega_1\lambda)\rightarrow \mathcal{P}(\Bbbk[\lambda]\slash\lambda^2)$. The edges $\operatorname{pc}(0, 1)\in\mathcal{P}(R\slash\omega_1\lambda)$ and $\operatorname{pc}(\lambda, 1)\in\mathcal{P}(R\slash\omega_1\lambda)$ project into distinct orbits in $\mathcal{P}(\Bbbk[\lambda]\slash\lambda^2)$, thus $\Omega'$ contains only one of them. Then, one can verify as before that $\mathcal{H}(\Omega')$ is uniquely determined in either of the two cases; so it ranges $I_1(9, 1)$ and $I_2(9, 1)$. Clearly, if $\omega_1\lambda = 0\in R$, then $\Omega = \Omega'$, hence this case is complete. Otherwise, the covering $\mathcal{P}(R)\rightarrow \mathcal{P}(R\slash\omega_1\lambda)$ is of degree 3. Then, one simply checks that all three edges in the preimage of $\operatorname{pc}(0, 1)\in\mathcal{P}(R\slash\omega_1\lambda)$ or $\operatorname{pc}(\lambda, 1)\in\mathcal{P}(R\slash\omega_1\lambda)$ are in the same orbit. As before, one can deduce that $\mathcal{H}(\Omega)$ is uniquely determined in both cases; so it ranges $I_1(9, 2)$ and $I_2(9, 2)$.

If $R$ is one of the exceptional rings of the previous paragraph, all orbits in $\mathcal{P}(R)$ are genus-zero. There are three orbits in $\mathcal{P}(\Bbbk[\lambda]\slash \lambda^3)$. As $\Omega$ ranges these orbits, $\mathcal{H}(\Omega)$ ranges $I_1(9, 1)$, $I_2(9, 1)$ and $H'(9,1)$. There are five orbits in $\mathcal{P}(\Bbbk[\lambda]\slash \lambda^4)$. As $\Omega$ ranges these orbits, $\mathcal{H}(\Omega)$ ranges $I_1(9, 2)$, $I_2(9, 2)$, $H_1'(9,2)$, $H_2'(9,2)$, and $H_3'(9,2)$. There is one orbit in $\mathcal{P}(\mathbb{Z}_9[\lambda]\slash(\lambda - 3k))$. Then, $\mathcal{H}(\mathcal{P}(\mathbb{Z}_9[\lambda]\slash(\lambda - 3k)))$ is equal to $H(9, 0)$ if $k=0$ and $H(9; a)$ with $a=-3k-1$ otherwise. There are three orbits in $\mathcal{P}(\mathbb{Z}_9[\lambda]\slash\langle3\lambda, \lambda^2\rangle)$. As $\Omega$ ranges these orbits, $\mathcal{H}(\Omega)$ ranges $I_1(9, 1)$, $I_2(9, 1)$ and $H(9,1)$. There are five orbits in $\mathcal{P}(\mathbb{Z}_9[\lambda]\slash\lambda^2)$. As $\Omega$ ranges these orbits, $\mathcal{H}(\Omega)$ ranges $I_1(9, 2)$, $I_2(9, 2)$, $H_1(9,2)$, $H_2(9,2)$, and $H_3(9,2)$.\qed


\emph{For a ring in which $\omega_1 = 0$, except the cases $R = \Bbbk$ and $R = \Bbbk[\lambda]\slash \lambda^2$, there is one orbit $\Omega_0$ in $\mathcal{P}(R)$ with $\mathcal{H}(\Omega_0) = I_1(3,1)$, one orbit $\Omega_1$ with $\mathcal{H}(\Omega_1) = I_2(3,1)$, and one has $\mathcal{H}(\Omega) = \tilde{I}(3)$ for any other orbit $\Omega$.} The proof is very similar to the case of $p = 5$. In summary, the only regions of size 1 and weight 1 are $\operatorname{pc}(0, 1)$ and $\operatorname{pc}(\lambda, 1)$ and these edges are in distinct orbits. Then, let $\Omega_0$ denote the orbit of $\operatorname{pc}(0, 1)$, and $\Omega_1$ denote the orbit of $\operatorname{pc}(\lambda, 1)$. One can easily verify as before that $\mathcal{H}(\Omega_0)$ and $\mathcal{H}(\Omega_1)$ are uniquely determined; so they are as given. In any other orbit $\Omega$, the regions are either of size 3 or of size 1 and weight $\frac{1}{3}$. Then, it is also easy to verify $\mathcal{H}(\Omega) = \tilde{I}(3)$ as in the case of wheels. As for the exceptional rings, there are three orbits in $\mathcal{P}(\Bbbk[\lambda]\slash \lambda^2)$. As $\Omega$ ranges these orbits, $\mathcal{H}(\Omega)$ ranges $I_1(3,1)$, $I_2(3,1)$, and $H(3,1)$. Finally, $\mathcal{H}(\mathcal{P}(\Bbbk)) = H(3,0)$.\qed

\subsection{The Table in the Main Theorem}
\label{sec:table}

\begin{table}[h!]
\small
\caption{\footnotesize Genus-zero $H := H(A)$ for $\mathfrak{m}$-local $A$ where $m=\langle p, t+1\rangle$ with $p\neq 2$}
\label{tab:groups}
\begin{center}
{\renewcommand{\arraystretch}{1.125}
\begin{tabular}{|l|c|c|c|}
\hline\hline
$H\subset\mathrm{Bu}_3$ & $A(H)$ & $\operatorname{c}(H)\subset\Gamma$ & $\operatorname{d}(H)$ \\
\hline\hline
$H_1(7, 1)$  & \multirow{3}{*}{$\mathbb{F}_7[t]\slash (t+1)^2$} & \multirow{4}{*}{$\Gamma_1(7)$} & $14$ \\
$H_2(7, 1)$  & & & $14$\\
$H_3(7, 1)$  & & & $14$\\
\cline{1-2}
$H(7, 0)$  & $\mathbb{Z}_7 \enskip \text{ where } \enskip t= -1$ & & $2$ \\
\hline\hline
$H_1(5, 1)$  & \multirow{2}{*}{$\mathbb{F}_5[t]\slash (t+1)^2$} & \multirow{4}{*}{$\Gamma_1(5)$} & $10$ \\
$H_2(5, 1)$  & & & $10$ \\
\cline{1-2}
$I(5, 1)$  & $\Lambda\slash (t^4-t^3+t^2-t+1)$ & & $10$ \\
\cline{1-2}
$H(5, 0)$  & $\mathbb{Z}_5 \enskip \text{ where } \enskip t= -1$ & & $2$ \\
\hline
$\tilde{I}(5)$  & $\Lambda^2 \slash (t^4-t^3+t^2-t+1)$ & \multirow{2}{*}{$\Gamma(5)$} & $10$ \\
\cline{1-2}
$\tilde{H}(5)$  & $\mathbb{Z}_5\oplus\mathbb{Z}_5 \enskip \text{ where } \enskip t= -1$ & & $2$ \\
\hline
$H(25; a)^*$  & $\mathbb{Z}_{25} \enskip \text{ where } \enskip t= a$ & $\Gamma_0(25)\cap\Gamma_1(5)$ & $10$ \\
\hline\hline
$H(3, 1)$ & $\mathbb{F}_3[t]\slash (t+1)^2$ & \multirow{4}{*}{$\Gamma_1(3)$} & $6$  \\
\cline{1-2}
$I_1(3, 1)$  & \multirow{2}{*}{$\Lambda\slash (t^2-t+1)$} & & $6$\\
$I_2(3, 1)$ & & & $6$ \\
\cline{1-2}
$H(3, 0)$  & $\mathbb{Z}_3 \enskip \text{ where } \enskip t= -1$ & & $2$ \\
\hline
$\tilde{I}(3)$  & $\Lambda^2 \slash (t^2-t+1)$ & \multirow{2}{*}{$\Gamma(3)$} & $6$\\
\cline{1-2}
$\tilde{H}(3)$  & $\mathbb{Z}_3\oplus\mathbb{Z}_3 \enskip \text{ where } \enskip t= -1$ & & $2$\\
\hline
$H(9; a)^*$  & $\mathbb{Z}_{9} \enskip \text{ where } \enskip t= a$ & $\Gamma_0(9)$ & $6$\\
\hline
$H_1(9, 2)$  & \multirow{3}{*}{$\mathbb{Z}_9[t]\slash (t+1)^2$} & \multirow{9}{*}{$\Gamma_1(9)$} & $18$ \\
$H_2(9, 2)$  & & & $18$\\
$H_3(9, 2)$  & & & $18$\\
\cline{1-2}
$I_1(9, 2)$  & \multirow{2}{*}{$\Lambda \slash \langle 3(t^2-t+1), (t^2-t+1)^2 \rangle$} & & $18$\\
$I_2(9, 2)$  & & & $18$\\
\cline{1-2}
$H(9, 1)$  & $\mathbb{Z}_9[t] \slash \langle 3(t+1), (t+1)^2 \rangle$ & & $6$ \\
\cline{1-2}
$I_1(9, 1)$  & \multirow{2}{*}{$\Lambda \slash \langle 3(t^2-t+1), (t^3+1) \rangle$} & & $6$\\
$I_2(9, 1)$  & & & $6$\\
\cline{1-2}
$H(9, 0)$  & $\mathbb{Z}_{9} \enskip \text{ where } \enskip t= -1$ & & $2$ \\
\hline
$H_1'(9, 2)$  & \multirow{3}{*}{$\mathbb{F}_3[t]\slash (t+1)^4$} & \multirow{4}{*}{$9J^0$} & $18$ \\
$H_2'(9, 2)$  & & & $18$ \\
$H_3'(9, 2)$  & & & $18$\\
\cline{1-2}
$H'(9, 1)$  & $\mathbb{F}_3[t]\slash (t+1)^3$ & & $6$\\
\hline
$\tilde{H}_1(9)$  & \multirow{2}{*}{$\mathbb{Z}_9\oplus\mathbb{Z}_3$ \enskip where \enskip $\tiny t= -\begin{pmatrix} 1 & 3\\ 0 & 1\end{pmatrix}$} & \multirow{2}{*}{$9H^0$} & $6$\\
$\tilde{H}_2(9)$ & & & $6$ \\ 
\hline
$H(27; a)^*$  & $\mathbb{Z}_{27} \enskip \text{ where } \enskip t= a$ & $27A^0$ & $18$ \\
\hline\hline
\end{tabular}}
\end{center}
\end{table}

The first column of Table~\ref{tab:groups} shows the names of the subgroups $H = H(A)$. Each subgroup must be understood to represent its conjugacy class. The subgroups which are marked with a $(^*)$ involve an additional parameter $a\in\mathbb{Z}_{p^k}$ such that $a\equiv -1\pmod p$ and $a\not\equiv -1\pmod {p^2}$ (note that $p^k = 25, 9, 27$). The second column shows $A(H)$, which is well-defined up to Burau equivalence when considered together with the distinguished epimorphism $\Lambda^2\twoheadrightarrow A(H)$, since $H$ is well-defined up to conjugacy. Note that several distinct subgroups correspond to (abstractly) isomorphic modules, but these are not Burau equivalent. In fact, for any $H$ in the table, one has $H = H(A(H))$ by Theorem~\ref{thm:setup}, item~\ref{thm:setup:closure}. The third column shows the subgroups $\operatorname{c}(H)\subset \Gamma$. In all of the cases in the table, $\operatorname{c}(H)$ is a congruence subgroup. However, this cannot be true in general: we know examples with other maximal ideals $\mathfrak{m}$ where $\operatorname{c}(H)$ is not a congruence subgroup. When there is no common notation for $\operatorname{c}(H)$, we use the notation of \citep{congruence}. Finally, the fourth column shows the depth $\operatorname{d}(H)$ (see Section~\ref{sec:braid}).

\bibliographystyle{plain}
\bibliography{ref}

\end{document}